\documentclass[11pt]{amsart}
\usepackage{amssymb}
\usepackage{curve2e}

\newtheorem{theorem}{Theorem}[section]
\newtheorem{lemma}[theorem]{Lemma}
\newtheorem{proposition}[theorem]{Proposition}
\newtheorem{corollary}[theorem]{Corollary}  

\newtheorem{definition}[theorem]{Definition}
\theoremstyle{definition}

\def \P {{\mathbb P}}
\newcommand{\E}{\mathbb E}

\newcommand{\HH}{\mathbb H}
\newcommand{\T}{\mathbb T}

\newcommand{\Q}{\mathbb Q}

\newcommand{\R}{\mathbb R}
\newcommand{\N}{\mathbb N}
\newcommand{\Z}{\mathbb Z}
\newcommand{\C}{\mathbb C}

\theoremstyle{remark}
\newtheorem{remark}[theorem]{Remark}

\numberwithin{equation}{section}
\begin{document}

\title[Matsumoto\,--Yor  process]{The Matsumoto and Yor  process and infinite dimensional hyperbolic space}

\author{ Philippe Bougerol}
\address{Laboratoire de Probabilit\'es et mod\`eles al\'eatoires,  
Universit\'e Pierre et Marie Curie, 4, Place Jussieu, 75005 Paris,  
FRANCE}
\email{philippe.bougerol@upmc.fr}


\date{}
\maketitle
\begin{center}
\textit{A la m\'emoire de Marc Yor, avec admiration. }
\end{center}

\begin{abstract}
The Matsumoto\,--Yor  process is $\int_0^t \exp(2B_s-B_t)\, ds$, where $(B_t)$ is a Brownian motion. It is shown that it is the limit of the radial part of the Brownian motion at the bottom of the spectrum on the hyperbolic space of dimension $q$, when $q$ tends to infinity. Analogous processes on infinite series of non compact symmetric spaces and on regular trees are described.
\end{abstract}

\setcounter{tocdepth}{2}
\tableofcontents

\setlength{\unitlength}{0.7mm}

\section{Introduction}

The aim of this paper is mainly to see in the first part that the Matsumoto\,--Yor  process \cite{Matsumoto_Yor_1}
$$ \eta_t=\int_0^{t} e^{2 B_s-B_t}\, ds, t \geq 0,$$
where $(B_t)$ is a standard Brownian motion, 
appears naturally as the radial part of the Brownian motion on the hyperbolic space $\HH_q$ at the bottom of the spectrum
when the dimension $q\to +\infty$. Marc Yor \cite{Yor} asked in 1999 whether there is a geometric interpretation of this process which provides a direct proof of its Markovianity. 
 Notice that $$\int_0^{t} e^{\mu B_s-B_t}\, ds, t \geqÊ0,$$ is a Markov process only for $\mu=1$ and $\mu=2$ (this follows from \cite{Matsumoto_Ogura} by a scaling and limit argument). The case $\mu=1$ is easy since it follows from Ito's formula that this process is solution of a stochastic differential equation. Things are more subtle for $\mu=2$ because the process is not Markov for the filtration $\sigma(B_s, 0 \leqÊs \leqÊt)$.
  
In a second part, we describe the limit of the radial part of the Brownian motion for the three infinite series of symmetric spaces of higher rank, namely $SO(p,q),$ $SU(p,q),$ $Sp(p,q)$ when $p$ is fixed and $q\to +\infty$. For $SU(p,q)$ we give the generator of the limit by studying the asymptotic behaviour of spherical functions (such an analysis is not available yet for the other cases).

In the higher rank case, these processes are different from the Whittaker processes obtained by O'Connell \cite{OConnell} and Chhaibi \cite{Chhaibi} who also generalize the Matsumoto\,--Yor  process for real split semi simple groups, but in a different (and more interesting) direction linked  with representation theory  and  geometric crystals.

In the last part we show that in the $q$-adic case of rank one, or more generally on regular trees $\T_q$, the radial part of the simple random walk at the bottom of the spectrum converges when $q\to \infty$ to the discrete Pitman process $$2\max_{0 \leq k \leqÊn}{\Sigma_k}-\Sigma_n,n \in \N,$$ where $(\Sigma_n)$ is the simple random walk on $\Z$. Contrary to the real case we use here only elementary arguments and the treatment is self-contained.

In an appendix we describe how to modify Shimeno \cite{Shimeno} in order to obtain the needed asymptotics for spherical functions in rank one. 
\section{The Matsumoto\,--Yor  process}
Let us first recall a now classical theorem of Pitman \cite{Pitman}.
 Let $B_t, t \geqÊ0,$ be a standard real Brownian motion starting at $0$. 
 \begin{theorem}[Pitman, \cite{Pitman}] The process
 $$2\max_{0 \leq s \leq t}B_s - B_t, t \geq 0,$$
 is a Markov process on $\R^+$. It is the Bessel(3) process with generator
 $$\frac{1}{2}\frac{d^2}{dr^2}+\frac{1}{r}\frac{d}{dr}.$$
 \end{theorem}
 In 1999, Matsumoto and Yor  \cite{Matsumoto_Yor_1, Matsumoto_Yor_2, Matsumoto_Yor_3}, have found the following exponential generalization of Pitman's theorem.
 \begin{theorem}[Matsumoto and Yor, \cite{Matsumoto_Yor_1}] The process $$\eta_t=\int_0^{t} e^{2 B_s-B_t}\, ds, t \geqÊ0,$$
is a Markov process and the generator of 
 $\log \eta_t$
is
  $$\frac{1}{2}\frac{d^2}{dr^2}+(\frac{d}{dr} \log K_0(e^{-r}))\frac{d}{dr},$$
where $K_0$ is the Macdonald function.
\end{theorem}
Recall that, for $\lambda\in \R$, the classical Macdonald function $K_\lambda$ is 
    $$K_{\lambda}(x)=
   \frac{1}{2}({\frac{x}{2}})^\lambda \int_0^{\infty}\frac{e^{-t}e^{-x^2/4t}}{t^{1+\lambda}} \, dt.$$
 A geometric intuition of the Markovianity of this process does not follow clearly from the rather intricate known proofs (either the original ones (\cite{Matsumoto_Yor_1, Matsumoto_Yor_2}) or Baudoin \cite{Baudoin}, see also \cite{Baudoin_OConnell}). Pitman's theorem can be recovered by Brownian scaling and Laplace's method.

  \section{Brownian motion on hyperbolic spaces}
  \subsection{Hyperboloid model} 
  There are many realizations of the hyperbolic spaces (see for instance Cannon et al.\,\cite{Cannon}). We will consider two of them: the hyperboloid and the upper half-space models. A very convenient reference for us is Franchi and Le Jan's book \cite{Franchi_LeJan}, which we will follow. A more Lie theoretic approach will be applied in Section \ref{section_higher} for the higher rank case (and of course could also be used here).
  
  For  $\xi,\xi' \in \R^{q+1}$ let
  $$\langle \xi, \xi' \rangle =\xi_0\xi'_0-\sum_{k=1}^q\xi_k\xi'_k.$$
  The hyperbolic space $\HH_{q}$ of dimension $q$  is 
  $$\HH_{q}=\{\xi\in \R^{q+1}; \langle \xi, \xi \rangle=1, \xi_0 >0\},$$
which is the upper sheet of an hyperboloid,  with the Riemannian metric $d$ defined by
  $$\cosh d(\xi,\xi')= \langle \xi, \xi' \rangle.$$
The group $SO_0(1,q)$ is the connected component of the identity in $$\{g\in Gl(q+1,\R); \langle g\xi, g\xi' \rangle = \langle \xi, \xi' \rangle, \xi, \xi' \in \R^{q+1}\}.$$  It acts by isometry on $\HH_{q}$ by matrix multiplication on $\R^{q+1}$  .
Let    $\{e_0,\cdots,e_q\}$ be the canonical basis of $\R^{q+1}$.  We consider $e_0$ as an origin in $\HH_{q}$ and we let $o=e_0$.   The subgroup $$K=\{g \in SO_o(1,q); ge_0=e_0\}$$ is isomorphic to $SO(q)$, it is a maximal compact subgroup of $SO_o(1,q)$ and $$\HH_{q}=SO_o(1,q)/K.$$ Each $\xi \in \HH_{q}$ can be written uniquely as 
  $$\xi=(\cosh r) e_0 + (\sinh r) \varphi,$$
  where $r \geq 0, \varphi \in S^{q-1}=\{\sum_{k=1}^q \varphi_ke_k; \sum_{k=1}^q \varphi_k^2=1\}$. One has  $d(o,\xi)=r$ and $(r,\varphi)$ are called the polar coordinates of $\xi$.    
The Laplace Beltrami operator $\Delta$ on $\HH_{q}$ is given in these coordinates by (see \cite[Prop.~3.5.4]{Franchi_LeJan})
\begin{equation}\label{LapBel}\Delta=\frac{\partial^2}{\partial r^2}+(q-1)\coth r \frac{\partial}{\partial r}+\frac{1}{\sinh^2 r}\Delta_{S^{q-1}}^{\varphi},\end{equation}
where $\Delta_{S^{q-1}}^{\varphi}$ is the Laplace operator on the sphere $S^{q-1}$ acting on the $\varphi$-variable. 

We denote $(\xi_t)$ the Riemannian Brownian motion on $\HH_{q}$,  defined as  the diffusion process with generator $\Delta/2$, starting from the origin $o=e_0$.
Since the group $K$ acts transitively on the spheres $\{\xi\in \HH_q; d(o,\xi)=r\}$,  $d(o, \xi_t)$ is a Markov process, called the radial part of $\xi_t$. It follows from (\ref{LapBel}) that its generator is $\Delta_R/2$ where \begin{equation}\label{DeltaR}\Delta_R=\frac{d^2}{d r^2}+(q-1) \coth r \frac{d}{dr}.\end{equation}

  \subsection{Upper half space model}  
 We introduce the upper half space model following Franchi and Le Jan \cite{Franchi_LeJan}:
 we consider the square matrices $\tilde E_j, 1 \leq j \leq 1+q$,  of order $q+1$ given by the expression
$$
t\tilde E_{1}+\sum_{j=1}^{q-1 }{x_{j}\tilde E_{j+1}}=\begin{pmatrix} 
0 & t  &  x_{1}  & x_{2} & .. & x_{q-1}\\
t & 0  &  x_{1} & x_{2} & .. & x_{q-1}   \\
x_{1} & -x_{1}  &  0 & 0 & .. & 0   \\
x_{2} & -x_{2}  &  0 & 0 & .. & 0   \\
.. & .. &  .. & .. &  ..   & .. \\
x_{q-1} & -x_{q-1} &  0 & 0 &  ..   & 0 \\
\end{pmatrix}, $$
when $t,x_1,\cdots,x_{q-1} \in \R.$
 Then
  $$\mathfrak S=\{t\tilde E_{1}+\sum_{j=1}^{q-1 }{x_{j}\tilde E_{j+1}}, t\in \R, x\in \R^{q-1}\}$$  is a solvable subalgebra of the Lie algebra of $SO_0(1,q)$. 
For $x\in \R^{q-1}$ and $y >0$ let
   $$T_{x,y}= \exp (\sum_{j=1}^{q-1} x_{j}\tilde E_{j+1}) \exp ((\log y) \tilde E_1).$$
 Then $$S=\{T_{x,y}; (x,y)\in \R^{q-1}\times \R_+^*\}$$
  is the Lie subgroup of $SO_o(1,q)$ with Lie algebra $\mathfrak S$. Moreover,
  $T_{x,y}T_{x',y'}=T_{x+yx',yy'}$
 (see \cite[Proposition 1.4.3]{Franchi_LeJan}). Therefore the map $(x,y)\mapsto T_{x,y}$ is an isomorphism between the affine group of $\R^{q-1}$, namely the semi-direct product  $\R^{q-1}\rtimes  \R_+^*$, and the group $S$. 
  One has (\cite[Proposition 2.1.3, Corollary 3.5.3]{Franchi_LeJan}),\begin{proposition}\label{dist}
The map $\tilde T: \R^{q-1}\times \R_+^* \to \HH_q$ given by 
$\tilde T(x,y)=T_{x,y}e_0$ is a diffeomorphism. In these so-called horocyclic or Poincar\'e coordinates $(x,y)\in \R^{q-1}\times \R_+^*$, the hyperbolic distance is given by
$$\cosh d(\tilde T(x,y),\tilde T(x',y'))=\frac{\|x-x'\|^2+y^2+y'^2}{2yy'}.$$
The pull back of the Laplace Beltrami operator $\Delta$ on $\HH_q$ is
$$y^2\frac{\partial^2}{\partial y^2}+(2-q)y\frac{\partial}{\partial y}+\Delta_{q-1}^x,$$
where $\Delta_{q-1}^x$ is the Euclidean Laplacian of $\R^{q-1}$ acting on the coordinate $x$.
\end{proposition}
 In horocyclic coordinates on $\HH_{q}$, the hyperbolic Brownian motion $\xi_t$ has a nice probabilistic representation which comes from the fact that it can be seen as a Brownian motion on the group $S$ where 
 \begin{definition}On a Lie group a process is called a Brownian motion if it is a continuous process with independent stationary multiplicative increments.
 \end{definition}
 Indeed, let $(X_t,Y_t)\in \R^{q-1} \times \R_+^*$ be the horocyclic coordinates of $\xi_t$, and let 
 $$\varrho=\frac{q-1}{2},$$
 then (see \cite{Bougerol}, or \cite[Theorem 7.6.5.1]{Franchi_LeJan}):   \begin{proposition}\label{prop_horo}
One can write $$ X_t=\int_0^t e^{B_s^{-\rho}} \, dW_s^{(q-1)}, \, Y_t=e^{B_t^{-\rho}},$$
  where $B_t^{-\rho}=B_t-t \varrho$ is a real Brownian motion on $\R$ with drift $-\varrho$ and $W_s^{(q-1)}$ is a $q-1$-dimensional standard Brownian motion, independent of $B^{-\varrho}$.
    \end{proposition}
  \proof This follows immediately from Ito's formula and the expression  of the generator given in Proposition  \ref{dist}.

  \subsection{Ground state processes}
  
  \subsubsection{Ground state processes on a manifold}
  We will need the notion of ground state process. In order to introduce it rapidly
  we use the set up presented in Pinsky \cite{Pinsky} (see also Pinchover \cite{Pinchover}). On a manifold $M$ we consider an elliptic operator which can be written locally as
$$D=\sum_{i,j}a_{ij}(x)\frac{\partial^2}{\partial x_i \partial x_j}+\sum_{i}b_{i}(x)\frac{\partial}{\partial x_i }+V(x),$$
where the coefficients are  $C^{\infty}$ and the matrix $a$ is symmetric positive definite (i.e. hypothesis $H_{loc}$  in \cite[p.~124]{Pinsky}).

Let 
$$C_D=\{u \in C^{\infty}(M); Du(x) =0, u(x) > 0, \mbox { for all } x \in M\}.$$
There exists $\lambda_0 \in (-\infty,+\infty]$ such that, for any $\lambda < \lambda_0$, $C_{D-\lambda}$ is empty,  and for $\lambda > \lambda_0$, $C_{D-\lambda}$ is not empty (\cite[Theorem 4.3.2]{Pinsky}). For a self adjoint operator $-\lambda_0$  coincide with the bottom of the spectrum on $L^2$ under general conditions (\cite[Proposition 4.10.1]{Pinsky}), for instance for the Brownian motion on a Riemannian manifold.
\begin{definition}
When $\lambda_0 < +\infty$, $\lambda_0$   is  the generalized principal eigenvalue of $D$.
\end{definition}
A positive  function $h \in C^{\infty}(M)$ such that $Dh=\lambda_0h$ is called a ground state (it does not always exist and is  in general not unique). The Doob $h$-transform of $D$ is the operator $D^h$ defined by
$$D^hf =\frac{1}{h}D(hf)-\lambda_0 f.$$
 The associated Markov process is called the $h$--ground state process. When $a$ is the identity matrix, then
\begin{equation}\label{D_h}D^h=D+2\sum_{i}\frac{\partial \log(h)}{\partial x_i }\frac{\partial}{\partial x_i }, \end{equation}
(see \cite[Section 4.1]{Pinsky}).
For the Brownian motion on a Riemannian manifold, when $V=0$ and once an origin is fixed, there is sometimes a canonical choice of a ground state for which the ground state process has a  probabilistic interpretation as an infinite Brownian loop. This infinite Brownian loop is, loosely speaking, the limit of the first half of the Brownian bridge around the origin when its length goes to infinity  (see \cite{Anker} for details).

  \subsubsection{A radial ground state process on $\HH_q$}
  
    In Proposition \ref{prop_horo} the component $Y_t$ of the Brownian motion on $\HH_q$ depends on the dimension $q$ only through the drift $\varrho=(q-1)/2$. In order to see what happens when $q \to +\infty$, it is natural  to first kill this drift. This is why we will consider the Brownian motion at the bottom of its spectrum.
  
The generalized principal eigenvalue of the Laplace Beltrami operator $\Delta$ is $-\varrho^2$
  and there exists a unique radial function $\varphi_0$ on $\HH_q$, called the basic Harish Chandra function,  such that 
  $$\Delta \varphi_0= -\varrho^2 \varphi_0$$ and 
  $\varphi_0(o)=1$ (see, e.g., Davies \cite[5.7.1]{Davies}, Gangolli and Varadarajan, \cite{Gangolli_Varadarajan}, Helgason \cite{Helgason}). 
 \begin{definition}
The infinite Brownian loop on $\HH_{q}$ is the $\varphi_0$--ground state process  $(\xi_t^0)$ with generator $$\frac{1}{2}\Delta^{\varphi_0} f = \frac{1}{2\varphi_0}\Delta (\varphi_0 f) +\frac{\varrho^2}{2}f.$$
  \end{definition}
  The function $\varphi_0$ is radial, which means that $\varphi_0(\xi)$ is  a function of $r=d(o,\xi)$; we define $\tilde \varphi_0:\R\to \R$ by, $$\tilde \varphi_0(r)=\varphi_0(\xi).$$
  It follows from (\ref{DeltaR}) and (\ref{D_h}) that,
  \begin{proposition}\label{prop_gen_bas}
   The process $\{d(o,\xi^0_t), t \geqÊ0\}$ is a Markov process on $\R^+$ with generator $\Delta^{\tilde \varphi_0}_R/2$ where
    $$\Delta^{\tilde \varphi_0}_R= \frac{d^2}{ dr^2}+(
(q-1)\coth r+2\frac{\tilde \varphi_0'(r)}{\tilde \varphi_0(r)}) \,  \frac{d}{ {dr}}.$$
    \end{proposition}
  
   \subsubsection{A non radial ground state process on $\HH_q$}
 
  Although our main interest is in $\xi_t^0$  we will actually need another ground state process for which computations are easier.  It is not invariant under rotations but its radial part is the same as the one of $\xi_t^0$ (see Proposition \ref{same_law} below).
Let $\Psi_0:\HH_q\to \R$ be the function defined by, if $(x,y)\in  \R^{q-1}\times \R_+^* $ are the horocyclic coordinates of $\xi \in \HH_q$,
  \begin{equation}\label{Psi_0}\Psi_0(\xi)=y^{(q-1)/2}=e^{\varrho \log y}. \end{equation}
 Notice that $\Psi_0(o)=1$ and
  $$\Delta \Psi_0 (\xi) =y^2\frac{\partial^2}{\partial y^2}y^\varrho-(q-2)y\frac{\partial}{\partial y} y^\varrho=-\varrho^2 \Psi_0(\xi).$$
 Thus  $
  \Psi_0$ is, like $\varphi_0$, a positive ground state of $\Delta$.  
 Let $\{S_t, tÊ\geqÊ0\}$ be the $\Psi_0$-ground state process of the Brownian motion $(\xi_t)$ on $\HH_q$. By definition, for all $T>0$, when $f:C([0,T],\HH_q)\to \R^+$ is measurable, 
  \begin{equation}\label{Psi0}\E(f(S_t, 0 \leqÊt \leqÊT))=e^{\varrho^2T/2}\E(f(\xi_t, 0 \leq t \leqÊT)\frac{\Psi_0(\xi_T)}{\Psi_0(\xi_0)})\end{equation}
where $\Psi_0(\xi_0)=1$ since $\xi_0=o$. 
  Then, it is easy to see that:
  \begin{lemma}[\cite{Bougerol}]
  In horospherical coordinates,    
  $$S_t=(\int_0^t e^{B_s} \, dW_s^{(q-1)}, e^{B_t}),$$
  where $B_t$ is a  standard (i.e. driftless) one dimensional Brownian motion, and $W^{(q-1)}$ is a standard $q-1$ dimensional Brownian motion, independent of $B$.
  \end{lemma}
Let $dk$ be the Haar measure on $K$, normalized as a probability measure. The function 
  $$\int_K \Psi_0(k\cdot \xi)\, dk$$
  is a positive radial eigenvector of $\Delta$ with eigenvalue $-\varrho^2$. Therefore, by uniqueness, we have the well known formula of Harish Chandra (see Helgason, \cite{Helgason})
  $$\varphi_0(\xi)=\int_K \Psi_0(k\cdot \xi)\, dk.$$
 The processes $S_t$ and $\xi_t^0$ do not have the same law, and $S_t$ is not rotation invariant. However,
  \begin{proposition}\label{same_law}
The two processes $\{d(o, S_t),t \geq 0\}$ and $\{d(o, \xi_t^0),t \geq0\}$
have the same law.  \end{proposition}
  \begin{proof}
    By invariance under rotation of the Brownian motion on $\HH_q$, for any $k\in K$, the processes $(\xi_t)$ and $(k\cdot \xi_t)$ have the same law.
  Hence, after integration over $K$, for $T>0$ and any measurable function $f:C([0,T],\R^+)\to \R^+,$ 
  \allowdisplaybreaks[4]
\hfill{}\begin{align*}
   \E(f(d(o,S_t), &0 \leqÊt \leqÊT))= e^{\varrho^2T/2}\E(f(d(o,\xi_t), 0 \leqÊt \leqÊT)\Psi_0(\xi_T))\\
         &= e^{\varrho^2T/2}\E(f(d(o,k\cdot \xi_t), 0 \leq t \leqÊT)\Psi_0(k\cdot \xi_T))\\
&= e^{\varrho^2T/2}\E(f(d(o,\xi_t), 0 \leq t \leqÊT)\Psi_0(k\cdot \xi_T))\\
      &= e^{\varrho^2T/2}\E(f(d(o,\xi_t), 0 \leq t \leqÊT)\int_K \Psi_0(k\cdot \xi_T)\, dk)\\
   &=e^{\varrho^2T/2}\E(f(d(o,\xi_t), 0 \leq t \leqÊT)\varphi_0(\xi_T))\\
   &=  \E(f(d(o,\xi_t^0), 0 \leqÊt \leqÊT)).   \end{align*}    \end{proof}
\begin{remark}\label{RemBusemann} In horocyclic coordinates in (\ref{Psi_0}) the function $(x,y)\mapsto\log y$ is the Busemann function on $\HH_q$ associated with the point at infinity $y=+\infty$ and $\Psi_0$ is a minimal eigenfunction. The process $(S_t)$ can be interpreted as the Brownian motion on $\HH_q$ conditioned to have 0 speed (i.e. $d(o,S_t)/t \to 0$ as $t \to +\infty$) and to exit at $y=+\infty$ (see \cite{Guivarch}). \end{remark}
\subsection{Matsumoto\,--Yor process as a limit}

 Our main result is the following.
  
  \begin{theorem}\label{theohyper}
  As $q\to +\infty$, the process $$d(0,  \xi_t^0)-\log q, t >0,$$converges in distribution to $\log \eta_t, t >0,$ where 
$$\eta_t=\int_0^{t} e^{2 B_s-B_t}\, ds$$   
is the Matsumoto\,--Yor  process, which is therefore a Markov process.
  \end{theorem}
  \begin{proof} By Proposition \ref{same_law}, it is enough to show that, almost surely,
 \begin{equation}\label{convd}\lim_{q \to \infty}Êd(o,S_t)-\log q = \log\int_0^{t} e^{2 B_s-B_t}\, ds.\end{equation}
 The origin $o$ in $\HH_q$ is $e_0=\tilde T(0,1)$. By  Proposition  \ref{dist},   $$\cosh d(o,\tilde T(x,y))=\frac{\|x\|^2+y^2+1}{2y},$$
 and
 $$ \cosh d(0,S_t)=\frac{\| \int_0^t e^{B_s} \, dW_s^{(q-1)}\|^2+e^{2B_t}+1}{2e^{B_t}},$$
thus
 $$ \frac{2\cosh d(0,S_t)}{q}=\frac{e^{B_t}+e^{-B_t}}{q}+e^{-B_t}\frac{1}{q}\sum_{k=1}^{q-1}(\int_0^t e^{B_s} \, d\beta_s^{(k)})^2,$$
  where $\beta_s^{(k)}, k\geq 1,$ are independent standard Brownian motions. Conditionally on the $\sigma$-algebra $\sigma(B_r, r \geqÊ0)$, the random variables $\int_0^t e^{B_s} \, d\beta_s^{(k)}, k\geq 1, $ are independent with the same distribution and 
 $$\E((\int_0^t e^{B_s} \, d\beta_s^{(k)})^2| \sigma(B_r, r\geqÊ0))=
\int_0^t e^{2B_s} \, ds.$$
 Therefore, by the law of large numbers, a.s.
 $$\lim_{q \to +\infty} \frac{1}{q}e^{d(0,S_t)}=\lim_{q \to +\infty} \frac{2}{q} \cosh d(0,S_t)=e^{-B_t}\int_0^t e^{2B_s} \, ds.$$
The limit is Markov as a limit of Markov processes. \end{proof}

Let us recover the generator of the Matsumoto\,--Yor  process. Let  $$
\delta(r) = \sinh^{q-1} r,
$$ then, by Proposition  \ref{prop_gen_bas},
$$\Delta^{\tilde \varphi_0}_R= \frac{d^2}{ dr^2}+ 2 \frac{d}{dr}\log (\delta^{1/2}\tilde  \varphi_0)(r) \frac{d}{ {dr}},$$
hence the generator of
 $d(o,\xi_t^0)-\log q$
 is
 \begin{equation}\label{gener}    \frac{1}{2}\frac{d^2}{ dr^2}+ \frac{d}{dr}\log (\delta^{1/2}\tilde \varphi_0)(r+\log q) \frac{d}{ {dr}}.
\end{equation} 
 Therefore the next proposition follows from Corollary \ref{cor_append} (since there, $m_\alpha=q-1$ for $SO(1,q)$).

\begin{proposition}[\cite{Matsumoto_Yor_1}]
The generator of the Matsumoto\,--Yor  process is   $$\frac{1}{2}\frac{d^2}{dr^2}+(\frac{d}{dr} \log K_0(e^{-r}))\frac{d}{dr}.$$
\end{proposition}
\begin{remark} We will clarify in Section \ref{section_higher} the occurrence of the ground state $K_0$ of the Toda operator $\frac{d^2}{ d{r^2}}-e^{-2r}$.
\end{remark}

  \subsection{A conditional law}
   
    The intertwining property which occurs in the proof of Matsumoto\,--Yor  theorem can also be establish by our approach.
\begin{proposition}[\cite{Matsumoto_Yor_1}]\label{condit}
    When $\eta_t= e^{-B_t}\int_0^{t} e^{2 B_s}\, ds$, and $\lambda \in \R$,
         $$\E(e^{\lambda B_t}| \sigma(\eta_s, 0 \leq s \leq t))=\frac{K_\lambda(1/\eta_t)}{K_0(1/\eta_t)}.$$    \end{proposition}
       \begin{proof}    For $\xi\in \HH_q$ with horocyclic coordinates $(x,y)\in  \R^{q-1}\times \R_+^* $ and $\lambda \in \R$, let
$\lambda(\xi)=\lambda \log y$.
 We consider the Harish Chandra function $\varphi_\lambda$ (see \cite[Theorem IV.4.3]{Helgason}) given by
$$\varphi_\lambda(\xi)=\int_K e^{(\lambda+ \rho)(k \cdot\xi)}dk,$$
(it of course depends on $q$) and we write $\tilde \varphi_\lambda(r)=\varphi_\lambda(\xi)$ when $r=d(o,\xi)$.
Since $\lambda(S_t)=\lambda B_t$ and $e^{\varrho(\xi)}=\Psi_0(\xi)$, it follows from (\ref{Psi0}) that, when $f:C([0,t]) \to \R^+$ is measurable,
\allowdisplaybreaks[4]
\begin{align*}
e^{-\varrho^2t/2}&\E(f(d(o, S_s), 0 \leq s \leq t))e^{\lambda B_t})\\
&=e^{-\varrho^2t/2}\E(f(d(o, S_s), 0 \leq s \leq t))e^{\lambda(S_t)})\\
&=\E(f(d(o,  \xi_s), 0 \leq s \leq t)e^{(\lambda+ \rho)( \xi_t)})\\
&=\E(f(d(o,  \xi_s), 0 \leq s \leq t)\int_K e^{(\lambda+ \rho)(k \cdot \xi_t)}dk)\\
&=\E(f(d(o,  \xi_s), 0 \leq s \leq t)\varphi_\lambda(\xi_t))\\
&=\E(f(d(o,  \xi_s), 0 \leq s \leq t)\frac{\varphi_\lambda(\xi_t)}{\varphi_0(\xi_t)}\varphi_0(\xi_t))\\
  &=\E(f(d(o,  \xi_s), 0 \leq s \leq t)\frac{\varphi_\lambda(\xi_t)}{\varphi_0(\xi_t)}e^{ \rho( \xi_t)})\\
  &=e^{-\varrho^2t/2} \E(f(d(o, S_s), 0 \leq s \leq t)\frac{\tilde \varphi_\lambda(d(o,S_t))}{\tilde \varphi_0(d(o,S_t))}  ).
\end{align*}
  We have again used the fact that for any $k\in K$, $\{\xi_t,t \geqÊ0\}$ has the same law as $\{k \cdot \xi_t,t \geqÊ0\}$. For each $q$,
\begin{align*}
  \E(f(d(0, S_s)-&\log q, 0 \leq s \leq t))e^{\lambda B_t})=\\ &\E(f(d(0, S_s)-\log q, 0 \leq s \leq t)\frac{\tilde\varphi_\lambda((d(o,S_t)-\log q) +\log q)}{\tilde\varphi_0((d(o,S_t)-\log q) +\log q)}  ).
  \end{align*}
  We choose for $f$ a continuous bounded function with bounded support. Letting $q \to +\infty$, it then follows from Theorem \ref{theohyper} and Corollary \ref{cor_append}  that
    $$\E(f(\log \eta_s, 0 \leq s \leq t)e^{\lambda B_t})=\E(f(\log \eta_s, 0 \leq s \leq t)\frac{K_{\lambda}(1/\eta_t)}{K_0(1/\eta_t)}).$$ \end{proof}
\begin{remark}\label{cameron} It is straightforward  to deduce from this proposition, Theorem \ref{theohyper} and Cameron and Martin's theorem  that, as proved by Matsumoto\,--Yor  \cite{Matsumoto_Yor_1},  when $B_t^{(\lambda)}=B_t+\lambda t$ is a Brownian motion with drift $\lambda$, then $\log \int_0^{t} e^{2 B_s^{(\lambda)}-B_t^{(\lambda)}}\, ds, t \geqÊ0,$
 is a Markov process on $\R$ with generator 
  $$\frac{1}{2}\frac{d^2}{dr^2}+(\frac{d}{dr} \log K_\lambda(e^{-r}))\frac{d}{dr}.$$
\end{remark}
 \section{Infinite series of symmetric spaces}\label{section_higher}
 
 We will now consider the same problem as above for the infinite series of symmetric spaces of non positive curvature, when the rank is fixed and the dimension goes to infinity.
  We will see that in the rank one case one finds the same result as for hyperbolic spaces, but that new phenomenons occur in higher rank.

 There are only three infinite series of (irreducible) Riemannian symmetric spaces of non positive (non zero) curvature namely
 the spaces $G/K$ where $$G=SO(p,q), SU(p,q) \mbox{ and } Sp(p,q),$$ and $K$ is a maximal compact subgroup (see \cite{Helgason}). We will suppose that $p \leqÊq$, then $p$ is the rank of the symmetric space.
 
  \subsection{$SU(p,q)$}
  We first consider the symmetric spaces associated with the series $SU(p,q)$ with the rank $p$ fixed, when $q\to +\infty$. 
 As in the case of hyperbolic spaces, it is convenient to use two descriptions of $SU(p,q)$.
  The usual description  is the following (see, e.g., \cite{Hoogenboom}, \cite{Knapp}): $SU(p,q)$ is the set of $(p+q)\times (p+q)$ matrices with entries in $\C$ and with determinant 1, which conserve the quadratic form
  $$\sum_{i=1}^p z_i\bar z_i- \sum_{j=1}^q z_{p+j}\bar z_{p+j} .$$ Let $I_p$ and $I_q$ be the identity matrix of order $p$ and $q$, and let $J=\begin{pmatrix} 
I_p &0  \\
0 & -I_q 
\end{pmatrix}$ then 
  $$SU(p,q)=\{M\in Sl(p+q,\C); M^*JM=J\}.$$
  If we write $M$ by block as
  $$M=\begin{pmatrix} 
M_1 &M_2  \\
M_3 & M_4 
\end{pmatrix},$$ where the size of the matrix  $M_1$ is $p\times p$, $M_2$ is $p\times q$, $M_3$ is $q\times p$ and $M_4$ is $q \times q$,
  then $M \in SU(p,q)$ when $M_1^*M_1-M_2^*M_2=I_p$, $M_4^*M_4-M_3^*M_3=I_q$, $M_1^*M_3=M_2^*M_4$, $\det(M)=1$.

  \subsection{Cartan decomposition and radial part}
Let $\bar K$ be the following maximal compact subgroup of $SU(p,q)$ (we put the notation $K$ aside for a later use) 
  $$\bar K=\{\begin{pmatrix} 
K_1 &0  \\
0 & K_4 
\end{pmatrix}\in Sl(p+q,\C), K_1 \in U(p), K_4 \in U(q)\}.$$
When $r=(r_1,\cdots,r_p)\in \R^p$, let
$$D_p(r)=\begin{pmatrix} 
\cosh r &       \sinh r      &  0  \\
\sinh r             & \cosh r   &0   \\
0             &              0   & I_{q-p} \\
\end{pmatrix},$$ 
where 
$$\cosh r =\begin{pmatrix} 
\cosh r_1 &              0 &\cdots  & 0             \\
0             & \cosh r_2  &\cdots  & 0              \\
0             & \cdots  &\cdots  & 0              \\
0             &              0  &\cdots & \cosh r_p  \\
\end{pmatrix},$$
and $\sinh r $ is the same matrix with $\cosh$ replaced by $\sinh$.
We consider the closed Weyl chamber
$$\mathfrak A^+=\{(r_1,\cdots,r_p)\in \R^p, r_1 \geqÊr_2 \geqÊ\cdots \geqÊr_p \geq 0\}.$$
The Cartan decomposition says that any $M$ in $SU(p,q)$ can be written as $M=k_1D_p(r) k_2$ with $k_1,k_2\in \bar K$ and $r\in \mathfrak A^+$. Such a $r$ is unique and is called the
 radial part of $M$. We let $r=Rad(M)$. Recall that,
 \begin{definition}\label{sing_val} Let $N$ be a $p \times p$ square matrix, the vector of singular values of $N$ is 
 $$\mbox{SingVal}(N)=(\sigma_1,\cdots,\sigma_p)\in \mathfrak A^+,$$ where the $\sigma_i$'s are the square roots of the eigenvalues of $NN^*$ written in decreasing order.
 \end{definition}
 \begin{lemma}\label{rad_expli} If $M=\begin{pmatrix} 
M_1 & M_2  \\
M_3 & M_4 
\end{pmatrix}\in SU(p,q)$, then $$\cosh \mbox{Rad}(M)=\mbox{SingVal}(M_1).$$ 
\end{lemma}
\begin{proof} By the Cartan decomposition, there exists $K_1,\tilde K_1\in U(p)$ and $K_4, \tilde K_4\in U(q)$ such that, if $r$ is the radial part of $M$ then
$$M=\begin{pmatrix} 
M_1 &M_2  \\
M_3 & M_4 
\end{pmatrix}=\begin{pmatrix} 
K_1 &0  \\
0 & K_4
\end{pmatrix}\begin{pmatrix} 
\cosh r  & \sinh r & 0  \\
\sinh r & \cosh r & 0\\
0 & 0 & I_{q-p} 
\end{pmatrix}\begin{pmatrix} 
\tilde K_1 &0  \\
0 & \tilde K_4
\end{pmatrix}.$$
Therefore $M_1=K_1(\cosh r)\tilde K_1$ and $M_1M_1^*=K_1(\cosh^2 r) K_1^*$ which proves the lemma.
\end{proof}

\medskip

  \subsection{Iwasawa decomposition and horocyclic coordinates}
It will be convenient to write matrices in $SU(p,q)$ in another basis to make more tractable the solvable part of the Iwasawa decomposition (see, e.g., Lu \cite{Lu}, Iozzi and Morris \cite{Iozzi_Morris} or Sawyer \cite{Sawyer}). We still suppose that $q \geq p$. Let 
$$P=\begin{pmatrix} 
I_p/\sqrt{2} &0& I_p/\sqrt{2}\\
I_p/\sqrt{2} & 0& -I_p/\sqrt{2}\\
0 & I_{q-p} & 0
\end{pmatrix}.$$ 
Then $P^*=P^{-1}$ and
$P^{-1} J P=\bar J$ where
$$\bar J=\begin{pmatrix} 
0 &0& -I_p\\
0 & I_{q-p}& 0\\
-I_p & 0 & 0
\end{pmatrix}.$$
We introduce
$$G=\{P^{-1}MP, M \in SU(p,q)\},$$
which is the set of matrices $M$ in $Sl(p+q,\C)$ such that $M^*\bar J M=\bar J$.
The group $G$ is obviously isomorphic to $SU(p,q)$. So, we will work with $G$ instead of $SU(p,q)$.

Let $l,b,c$ be  $p\times p, p \times (q-p)$ and $p\times p$ complex matrices, respectively, where $l$ is {\bf lower} triangular and invertible,
and define \begin{equation}\label{S_(_)}S(l,b,c)=\begin{pmatrix} 
l &b & cl^{*-1}\\
0 & I_{q-p} & b^*l^{*-1}\\
0 & 0 & l^{*-1}
\end{pmatrix}.\end{equation}
Let 
\begin{align*}S&=\{S(l,b,c) ; l \mbox{ is lower triangular with positive diagonal}, c+c^*=bb^*\},\\
A&=\{S(l,0,0) \in S; l \mbox{ is diagonal with positive diagonal}\},\\
N&=\{S(l,b,c) \in S; \mbox{ the diagonal elements of } l  \mbox{  are equal to } 1\},\\
K&=P^{-1}\bar KP.
\end{align*}
Notice that in general $l$ is lower triangular and hence neither $N$ nor $S$ is made of upper triangular matrices. If $D=\{D_p(r), r \in \R^p\}$ then $A=P^{-1}DP$. The Iwasawa decomposition is  $G=N  A  K$,   $N$ is the nilpotent component and $S=NA$. Notice that $SU(p,q)/\bar K$ is isomorphic with $G/K$. For $M\in G$, $PMP^{-1}$ is in $SU(p,q)$ and we let $Rad(M)=Rad(PMP^{-1})$. In $G/K$, we choose as origin $o=K$, then $Rad(M)$ plays the role of a generalized distance between the cosets $o=K$ and $\xi=MK$ in $G/K$. \begin{lemma}[\cite{Lu}]\label{lemLu}
For $S(l,b,c) \in S,$
$$\cosh \mbox{Rad}(S(l,b,c))=\frac{1}{2}\mbox{SingVal}(l+l^{*-1}+cl^{*-1}). $$
\end{lemma}
\begin{proof} This follows immediately from Lemma \ref{rad_expli} since, for $S(l,b,c)\in G$, the corresponding element in ${SU}(p,q)$ can be written as
$$PS(l,b,c)P^{-1}=\begin{pmatrix} 
M_1 &M_2  \\
M_3 & M_4 
\end{pmatrix},$$ where $M_1=(l+(I+c)l^{*-1})/2$.\end{proof}
 By the Iwasawa decomposition each element $\xi \in G/K$ can be written uniquely as $\xi=S(l,b,c) K$ with $S(l,b,c)\in S$. We see that we can and will identify $S$ and $G/K$. 
We call $S(l,b,c)$ the horocyclic coordinates of $\xi$. They generalize the horocyclic--Poincar\'e coordinates in $\HH_q$. 

\subsection{Brownian motion on $G/K$ and infinite Brownian loop} The Lie algebra $\mathfrak S$ of $S$ is \begin{align*}\mathfrak S=\{\begin{pmatrix} 
l &b & c \\
0 & 0 & b^*\\
0 & 0 & -l^*
\end{pmatrix}, & \, \mbox{where } c \mbox{ is skew}-\mbox{Hermitian}, l \mbox{ is lower triangular with} \\  &  \mbox{ real  diagonal}\}.\end{align*}
The Killing form on the Lie algebra of $G$ allows to define a scalar product on $\mathfrak S$ by
 $$<X,Y>=\frac{1}{2}\mathrm{Trace}(XY^*).$$
Then $\mathfrak S$ is a real Euclidean space. 
 Let $\mathfrak{A}$ and $\mathfrak{N}$ be the Lie algebras of $A$ and $N$,  then
 $\mathfrak S=\mathfrak{ A\oplus N}$.  Let $X_1,\cdots,X_p$ be an orthonormal basis of $\mathfrak A$ and $N_1,\cdots, N_s $ be an orthonormal (real) basis of $\mathfrak N$ adapted to the root space decomposition. 
 In horospherical coordinates the Laplace Beltrami operator on $G/K= S$ is (e.g. \cite[proof of Proposition 2.2]{Bougerol}, \cite{Cowling}, \cite[p.105]{Guivarch})
 $$\Delta= \sum_{i=1}^p  X_i^2+2\sum_{j=1}^s N_j^2-2\sum_{i=1}^p \varrho(X_i)X_i,$$
where $\varrho$ is given below by (\ref{rho}) and $X_i$ and $N_j$ are considered as left invariant vector fields.
 We consider the Riemannian Brownian motion $\{\xi_t, t \geqÊ0\}$ on $G/K= S$. It is the process with generator $\Delta/2$ starting from the origin $o$. 
 
 As in the hyperbolic case one can consider the ground state process $\{\xi^0_t, t \geqÊ0\}$ of this Brownian motion associated with the basic Harish Chandra spherical function $\varphi_0$. By \cite[Theorem IV.4.3]{Helgason}, 
\begin{equation}\label{eq_phi0}\varphi_0(g)=\int_{ K}e^{\varrho(H(kg))}\,dk,\end{equation} where for $g\in G$, we write $H(g)$ for the element of the Lie algebra $\mathfrak{A}$ of $A$ such that $g\in N e^{H(g)} K$ in the Iwasawa decomposition $G=N  A K$. The generator of $\xi^0$ is $\Delta^{\varphi_0}/2$ and it corresponds to the infinite Brownian loop on $G/K$  (see \cite{Anker}). 
By invariance of $\Delta$ under $K$, the radial part of $\mbox{Rad}(\xi_t)$ and $\mbox{Rad}(\xi_t^0)$ are Markov processes with values in the closed Weyl chamber $\mathfrak A^+$.
 
\subsection{Distinguished Brownian motion on $S$} We define $\Psi_0:G \to \R^+$ by
$$\Psi_0(g)=e^{\varrho(H(g))}.$$
Since $H(gk)=H(g)$ for any $k\in K$, $\Psi_0$ is well defined on $G/K$.
We consider the $\Psi_0$-ground state process $\{S_t, t \geq 0\}$ of the Brownian motion $\xi$ on $G/K$. Using the identification $G/K=S$, we see it as a process on $S$ (it has the same interpretation as the one given in Remark \ref{RemBusemann}). The following definition is used in harmonic analysis (see, e.g., \cite[Proposition 1.2]{Bougerol} , \cite{Cowling,Cowling_Giulini_Hulanicki_Mauceri}),
\begin{definition}
The distinguished Brownian motion on $S$ is the process $S_t,t\geq 0$.
\end{definition}
 The generator of $(S_t)$ is 
$$\frac{1}{2}(\sum_{i=1}^p  X_i^2+2\sum_{j=1}^s N_j^2).$$
One shows as in the hyperbolic case that (see also \cite[Proof of Theorem 6.1]{Anker}),
\begin{lemma}\label{memeloi} The two processes 
$\{\mbox{Rad}(S_t), t \geq 0\}$ and $\{\mbox{Rad}(\xi^{(0)}_t), t \geq0\}$ have the same law. 
\end{lemma} 
The process $(S_t)$ is a solution of a  stochastic differential equation. Indeed,  consider the Brownian motion $(V_t)$ on the Lie algebra $\mathfrak S$, considered as an Euclidean space,
$$ V_t=\begin{pmatrix} 
\lambda_t &\beta_t & \kappa_t \\
0 & 0 & \beta_t^*\\
0 & 0 & -\lambda_t^*\\
\end{pmatrix} \in \mathfrak S,$$
where  the coefficients   $\lambda_t^{r,r}, i\kappa_t^{r,r}/2, 1 \leqÊr \leq p$, and  the real and imaginary parts of 
$\lambda_t^{r,s}/\sqrt{2}, r > s$, $ \beta_t^{k,l}/\sqrt{2}, 1\leq k \leq p, 1\leq l \leq q-p$   and $\kappa_t^{r,s}/\sqrt{2}, 1Ê\leq r < s \leq q$ are standard real independent Brownian motions,  $\lambda_t^{r,s}=0$ if $1 \leq r < s \leq p$ and $\kappa_t^{r,s}=-\bar \kappa_t^{s,r}$, when $q \geq r > s \geqÊ1$ (we use $m^{i,j}$ to denote the $(i,j)$ coefficient of a matrix $m$).

When $(X_t)$ is a continuous semimartingale, we use $\delta X_t$ for its Stratonovich differential and $dX_t$ for its Ito one (see, e.g., \cite{Ikeda_Watanabe}).
\begin{proposition} The distinguished Brownian motion $(S_t)$ is the solution of the following Stratonovich stochastic differential equation in the set of $(p+q)\times (p+q)$ complex matrices,
$$ \delta S_t=S_t \, \delta  V_t, S_0=I_{p+q}.$$
\end{proposition}
 In order to compute the radial component we use a decomposition of $S$ which is slightly different from the factorization $S= N A$ coming from the Iwasawa factorization.
We write, using notation (\ref{S_(_)})
$$S_t=S(l_t,b_t,c_t),$$
and
 \begin{equation}\label{decS} S_t=M_t   L_t,\end{equation} where $$M_t=\begin{pmatrix}  
I_p &b_t & c_t \\
0 & I_{q-p} & b^*_t\\
0 & 0 & I_p
\end{pmatrix}, \,\,\,  L_t=\begin{pmatrix} 
l_t&0& 0 \\
0 & I_{q-p} & 0\\
0 & 0 & l_t^{*-1}
\end{pmatrix}.$$
Recall that the matrix $l_t$ is lower triangular. Its diagonal is the $A$ component of $S_t$ but it also has a part of the $N$ component. 
We solve the equation satisfied by $(S_t)$. By Stratonovich calculus,  $$ \delta S_t=M_t \, \delta  L_t+(\delta M_t )  L_t,$$
therefore, $$M_t  L_t \, \delta V_t=M_t \, \delta  L_t+( \delta M_t)   L_t,$$
which implies that 
$$  \delta V_t=L_t^{-1} \, \delta  L_t+   L_t^{-1} M_t^{-1}(\delta M_t)   L_t,$$
hence
$$  L_t^{-1} \, \delta  L_t=\,  \begin{pmatrix} 
\delta \lambda_t &0 & 0 \\
0 & 0 & 0\\
0 & 0 & -\delta \lambda_t^*\\
\end{pmatrix},$$
and
$$   L_t^{-1} M_t^{-1}(\delta M_t )  L_t=  \begin{pmatrix} 
0 &\delta \beta_t & \delta \kappa_t \\
0 & 0 &\delta \beta_t^*\\
0 & 0 & 0\\
\end{pmatrix}.$$
We obtain that
$ \delta l_t= l_t \, \delta \lambda_t$
and
$$\, \delta M_t= M_t  L_t \,  \begin{pmatrix} 
0 &\delta \beta_t &\delta \kappa_t \\
0 & 0 &\delta \beta_t^*\\
0 & 0 & 0\\
\end{pmatrix}   L_t^{-1},$$
which gives
$$\begin{pmatrix} 
0 &\delta b_t & \delta c_t \\
0 & 0 & \delta b^*_t\\
0 & 0 & 0
\end{pmatrix}= \begin{pmatrix} 
I_p &b_t & c_t \\
0 & I_{q-p} & b^*_t\\
0 & 0 & I_p
\end{pmatrix}\begin{pmatrix} 
0&l_t\delta \beta_t & l_t\delta  \kappa_t l_t ^*\\
0 & 0 & \delta \beta_t^*l_t^*\\
0 & 0 & 0
\end{pmatrix}.$$
This show that:
\begin{proposition} \label{express} 
$$\, \delta l_t= l_t \, \delta \lambda_t, b_t=\int_0^t l_s\, \delta \beta_s, c_t=\int_0^t l_s( \delta  \kappa_s) l_s^*+\int_0^t b_s ( \delta \beta_s^*)l_s^*.$$
\end{proposition}
In particular,
\begin{corollary}
The process $(l_t)$ is a Brownian motion on the subgroup of $Gl(p,\C)$ consisting of lower triangular matrices with positive diagonal. 
\end{corollary}
The linear equation  $\delta l_t= l_t \, \delta \lambda_t$ is therefore easy to solve explicitly by induction on $p$. For instance, when $p=2$, if  $$\lambda_t=\begin{pmatrix} 
\lambda_t^{(1)}&0\\
\lambda_t^{(3)} & \lambda_t^{(2)}
\end{pmatrix},$$
we obtain that $$l_t=\begin{pmatrix} 
e^{\lambda_t^{(1)}}&0\\
e^{\lambda_t^{(2)}}\int_0^t e^{\lambda_s^{(1)}-\lambda_s^{(2)}}\, \delta\lambda_s^{(3)} & e^{\lambda_t^{(2)}}
\end{pmatrix}.$$

\subsection{Limit as $q \to +\infty$} We study the asymptotic behaviour of the radial part $Rad(S_t)$ of the distinguished Brownian motion on $S$, using the decomposition (\ref{decS}). We first consider Ito's integral.
\begin{lemma} Almost surely,
$$\lim_{q \to +\infty} \frac{1}{q}\int_0^t [\int_0^s l_u d \beta_u] ( d \beta_s^*)l_s^* =0.$$
\end{lemma}
\begin{proof} For $1 \leq i,j\leq p$, 
$$[\int_0^t (\int_0^s l_u d \beta_u) ( d \beta_s^*)l_s^*]^{i,j}= \sum_{r=1}^{q-p}\sum_{n=1}^p \sum_{m=1}^p\int_0^t (\int_0^s  l_u^{i,n} d \beta^{n,r}_u) 
\bar l_s^{j,m} d \bar \beta^{m,r}_s.
$$
We fix $i,j,n,m$. Conditionally on the sigma-algebra $\sigma(l_s, s \geqÊ0)$, the random variables 
$$\int_0^t (\int_0^s  l_u^{i,n} d \beta^{n,r}_u) 
\bar l_s^{j,m} d \bar \beta^{m,r}_s$$ for $r=1,2,\cdots,$ are independent with the same law, and with expectation equal to $0$ since they are martingales. Therefore, by the law of large numbers,
$$ \frac{1}{q-p}\sum_{r=1}^{q-p}\int_0^t (\int_0^s  l_u^{i,n} d \beta^{n,r}_u) 
\bar l_s^{j,m} d \bar \beta^{m,r}_s$$
converges a.s. to $0$ when $q\to +\infty$, which proves the lemma.  \end{proof}
\begin{proposition} Let $S_t,t \geq 0,$ be the  distinguished Brownian motion. Then, a.s., 
$$\lim_{q \to +\infty}\frac{1}{q}\cosh \mbox{Rad}(S_t) = \mbox{SingVal}(l^{-1}_t\int_0^t l_sl^{*}_s  \, ds).$$
\end{proposition}
\begin{proof} By Proposition \ref{express},
$$c_t=\int_0^t l_s( \delta  \kappa_s) l_s^*+\int_0^t (\int_0^s l_u\, \delta  \beta_u) ( \delta \beta_s^*)l_s^*.$$Since the processes $(l_t)$ and $(\kappa_t)$ do not depend on $q$,
$$\lim_{q\to +\infty} \frac{1}{q} c_t= \lim_{q\to +\infty} \frac{1}{q}\int_0^t (\int_0^s l_u\, \delta  \beta_u) (\delta \beta_s^*)l_s^*.$$
Now, recall the link between Stratonovich and Ito integral: if $X$ and $Y$ are continuous semimartingales,
$$\int_0^t Y  \delta X= \int_0^t Y  dX +\frac{1}{2}\langle X, Y \rangle_t,$$
where, if $X$ and $Y$ are matrices, $\langle X, Y \rangle_t$ is the matrix with $(i,j)$ entrie $$\langle X, Y \rangle_t^{i,j}=\sum_{k} \langle X^{i,k}, Y^{k,j} \rangle_t.$$
Since $(\beta_t)$ is independent of $(l_t)$,
$\int_0^s l_u \delta  \beta_u= \int_0^s l_u \, d \beta_u,$
and
$$\int_0^t [\int_0^s l_u \delta  \beta_u] ( \delta \beta_s^*)l_s^*=
\int_0^t [\int_0^s l_u d \beta_u] ( d \beta_s^*)l_s^* +\frac{1}{2} \langle \int_0^t l_s d\beta_s, \int_0^t d\beta_s^* l_s^* \rangle_t,$$
so it follows from the preceding lemma that
$$\lim_{q \to +\infty} \frac{1}{q} c_t=\frac{1}{2} \lim_{q \to +\infty} \langle \int_0^t l_s d\beta_s, \int_0^t d\beta_s^* l_s^* \rangle_t.$$
On the other hand, for $1 \leqÊi,j \leqÊp$,
\begin{align*} \langle \int_0^t l_s &d\beta_s, \int_0^t d \beta_s^* l_s^* \rangle_t^{i,j}\\
&= \sum_{r=1}^{q-p}\langle (\int_0^t l_s d\beta_s)^{i,r}, (\int_0^t  d \beta_s^*l_s^* )^{r,j}
\rangle_t\\
&= \sum_{r=1}^{q-p}\sum_{k=1}^p  \sum_{l=1}^p
\langle
\int_0^t   l_s^{i,k} d \beta^{k,r}_s, \int_0^t 
\bar l_s^{j,l} d \bar \beta^{l,r}_s \rangle_t
\\
&= \sum_{r=1}^{q-p}\sum_{k=1}^p  \langle \int_0^t   l_s^{i,k} d \beta^{k,r}_s, \int_0^t 
\bar l_s^{j,k} d \bar \beta^{k,r}_s \rangle_t
\\
&= \sum_{r=1}^{q-p}\sum_{k=1}^p  \int_0^t    l_s^{i,k} \bar  l_s^{j,k} d\langle  \beta^{k,r},  \bar \beta^{k,r} \rangle_s.
\end{align*}
Recall that $\beta_t^{k,r}/\sqrt{2}$ is a complex Brownian motions, hence
$$\langle  \beta^{k,r},  \bar \beta^{k,r} \rangle_s=  4s,$$
 therefore
 $$ \sum_{r=1}^{q-p}\sum_{k=1}^p  \int_0^t    l_s^{i,k} \bar  l_s^{j,k} d\langle  \beta^{k,r},  \bar \beta^{k,r} \rangle_s=4(q-p)\sum_{k=1}^p  \int_0^t    l_s^{i,k} \bar  l_s^{j,k} \, ds,$$ and
 $$\langle \int_0^t l_s \,d\beta_s, \int_0^t d \beta_s^* l_s^* \rangle_t=4(q-p)\int_0^t l_sl^{*}_s \, ds.$$
Consequently, a.s.,  
$$\lim_{q \to +\infty} \frac{1}{q} c_t=\lim_{q\to +\infty}Ê\frac{2(q-p)}{q}\int_0^t l_sl^{*}_s \, ds=
2\int_0^t l_sl^{*}_s \, ds,$$ and the proposition follows from Lemma \ref{lemLu} and from the equality $\mbox{SingVal}(\int_0^t l_sl^{*}_sl_t^{*-1} \, ds)=\mbox{SingVal}(l_t^{-1}\int_0^t l_sl^{*}_s \, ds)$.
\end{proof} 

The following is a direct consequence of the previous proposition and Lemma \ref{memeloi}. Let ${\bf 1}\in \R^p$ denote the vector ${\bf 1}=(1,1,\cdots,1)$.
\begin{theorem}\label{SUpqmarkov}  For $SU(p,q)$, the process $\eta_t, t \geqÊ0,$ with values in  $\{(r_1,\cdots,r_p)\in \R^p, r_1 \geqÊr_2 \geqÊ\cdots \geqÊr_p\}$ given by 
$$\eta_t=\mbox{SingVal}(l^{-1}_t\int_0^t l_sl^{*}_s  \, ds)$$
is a Markov process.  
The process $$\mbox{Rad}(\xi_t^0)- \log (2q){\bf 1}, t \geqÊ0, $$ converges in distribution to $\log \eta_t, t \geqÊ0.$
\end{theorem}
We will compute  the generator of $(\eta_t)$  in  Theorem \ref{Theo_gen_lim}. 
\subsection{$SO(p,q)$ and $Sp(p,q)$}
For $SO(p,q)$ the only difference with $SU(p,q)$ is that all the entries are real. Hence, in the preceding computation, 
$$\langle  \beta^{k,r},  \bar \beta^{k,r} \rangle_s= \langle  \beta^{k,r},   \beta^{k,r} \rangle_s = 2s.$$
Similarly, for $Sp(p,q)$ the entries are quaternionic, therefore, in that case,
$$\langle  \beta^{k,r},  \bar \beta^{k,r} \rangle_s = 8s.$$
So we obtain,
\begin{theorem}\label{Spqmarkov}  For $SO(p,q)$, resp. $Sp(p,q),$ in distribution, $$\lim_{q \to +\infty} \mbox{Rad}(\xi_t^0)- \log (\theta q) {\bf 1} = \log \mbox{SingVal}(l^{-1}_t\int_0^t l_sl^{*}_s  \, ds),$$ with $\theta=1$ for $SO(p,q)$, resp. $\theta=4$ for $Sp(p,q)$, and where $l$ is the Brownian motion on the group of $p\times p$ lower triangular matrices with real, resp. quaternionic, entries and with positive diagonal,  solution of $\delta l_t =l_t \delta \lambda_t$.
\end{theorem}

\begin{corollary} In rank one, i.e. $p=1$, $$\lim_{q \to +\infty} \mbox{Rad}(\xi_t^0)- \log (\theta  q) {\bf 1} = \log \int_0^t e^{2B_s-B_t}\, ds,$$
where $B$ is a standard Brownian motion.
\end{corollary}

\section{Generator of $(\eta_t)$}

\subsection{Inozemtsev limit to Quantum Toda Hamiltonian}\label{Inozemtsev} In order to compute the generator $L$ of the process $(\eta_t)$ process in Theorem \ref{SUpqmarkov} it is convenient to relate $Rad(\xi_t^0)$ to a Calogero Moser Sutherland model. 

We consider the general case of a symmetric space $G/K$ associated with one of the groups $G=SO(p,q), SU(p,q)$ and $Sp(p,q)$, where $K$ is a maximal compact subgroup. The root system is of  type $BC_p$.  Recall that, if $\{e_1,\dots,e_p\}$ is an orthonormal basis of 	a real Euclidean space $\mathfrak{a}$ of dimension $p$, with dual basis $(e_i^*)$ then the positive roots $\Sigma^+$ of the root system $BC_p$ are given by
$$\Sigma^+=\{ e_k^*, 2e_k^*, (1 \leq k \leq p), e_i^*+e_j^*, e_i^*-e_j^*, (1 \leqÊi < j \leq p)\}.$$
We associate to each considered group a triplet $m=(m_1,m_2,m_3)$ given by, when $p \geqÊ2$, for $SO(p,q), m=(q-p,0,1)$, for $SU(p,q), m=(2(q-p),1,2)$ and for $Sp(p,q), m= (4(q-p),3,4)$. When $p=1$, the only difference is that $m_3=0$.
 The quantum Calogero-Moser-Sutherland trigonometric Hamiltonian of type $BC_p$ is (see \cite{ Olshanetsky_2, Oshima_Shimeno})
\begin{align*}\label{eqn:sutherlandBC}
H_{CMS}=&H_{CMS}(p,q;r)=\sum_{k=1}^p (\frac{\partial^2}{ \partial{r_k^2}}-\frac{m_1(m_1+2m_2-2)}{4 \sinh^2 r_k}-\frac{m_2(m_2-2)}{ \sinh^2 2r_k})\\
&-\sum_{1 \leqÊi < j \leq p}m_3(m_3-2)(\frac{1}{2\sinh^2(r_i-r_j)}+\frac{1}{2\sinh^2(r_i+r_j)}).
\end{align*}
For $\alpha \in \Sigma^+$ we let  $m_\alpha=m_1$ if $\alpha=e_k^*$, $m_\alpha=m_2$ if $\alpha=2e_k^*$ and $m_\alpha=m_3$ if $\alpha=e_i^*+e_j^*$ or $\alpha= e_i^*-e_j^*.$ Define \begin{equation}\label{rho}\rho=\frac{1}{2}\sum_{\alpha \in \Sigma^+}m_\alpha\, \alpha,\,\, \delta=\prod_{\alpha\in \Sigma_+} (e^\alpha-e^{-\alpha})^{m_\alpha}.  \end{equation}
The radial part $\Delta_R$ of the Laplace Beltrami operator on $G/K$ is 
equal  (see \cite[p.268]{Helgason}) to $$\Delta_R=\sum_{k=1}^p \frac{\partial^2}{ \partial{r_k^2}}+ \sum_{k=1}^p \frac{\partial}{ \partial{r_k}}(\log \delta)\frac{\partial}{ \partial{r_k}},$$
 which can be written as 
$$\label{eqn:sutherland}
\Delta_R=\delta^{-1/2}\circ \{H_{CMS}  -(\rho,\rho)\} \circ \delta^{1/2}$$
 (see \cite{Olshanetsky_1, Olshanetsky_2, Oshima_Shimeno}). On the other hand, we define $\tilde \varphi_0:\R^p\to \R$ by $\tilde \varphi_0(r)=\varphi_0(D_p(r))$ where $\varphi_0$ is the basic Harish Chandra function. Then $\tilde \varphi_0(0)=1$ and
$$\Delta_R\tilde \varphi_0=-(\rho,\rho)\tilde \varphi_0.$$
Therefore
$$H_{CMS} (\delta^{1/2}\tilde \varphi_0)=0.$$
The generator of the radial part $\mbox{Rad}(\xi_t^0)$ of the ground state process $(\xi_t^0)$ associated with $\varphi_0$ is $\Delta^{\tilde \varphi_0}_R/{2}$ where
\begin{align*}\label{deltaetH}{\Delta^{\tilde \varphi_0}_R} &=
\sum_{k=1}^p \frac{\partial^2}{ \partial{r_k^2}}+ 2\sum_{k=1}^p \frac{\partial}{ \partial{r_k}}(\log \delta^{1/2} \tilde \varphi_0)\frac{\partial}{ \partial{r_k}}\\ &= (\delta^{1/2}\tilde \varphi_0)^{-1} \circ H_{CMS} \circ (\delta^{1/2}\tilde \varphi_0).\end{align*}
 For $SU(p,q)$, since $m_3=2$, we have the following, called the  Inozemtsev limit (\cite{Inozemtsev}),
$$\lim_{q \to +\infty} H_{CMS}(p,q;r+\log (2q) {\bf 1})=\sum_{k=1}^p \frac{\partial^2}{ \partial{r_k^2}}-e^{-2r_k}.$$
Similarly, 
for $SO(p,q)$,
$$\lim_{q \to +\infty} H_{CMS}(p,q;r+\log q\, {\bf 1})=\sum_{k=1}^p ( \frac{\partial^2}{ \partial{r_k^2}}-e^{-2r_k})
+\sum_{1 \leqÊi < j \leq p}\frac{1}{2\sinh^2(r_i-r_j)},$$
and for $Sp(p,q)$,
$$\lim_{q \to +\infty} H_{CMS}(p,q;r+\log (4q) {\bf 1})= \sum_{k=1}^p (\frac{\partial^2}{ \partial{r_k^2}}- e^{-2r_k})
-\sum_{1 \leqÊi < j \leq p}\frac{4}{\sinh^2(r_i-r_j)}.$$
We denote by $H_{T}$ the respective limit. 
Since $\delta^{1/2}\tilde \varphi_0$ is a ground state of $H_{CMS}$, it is reasonable to infer that there is a ground state $\Psi$ of $H_{T}$ such that the generator $L$ of $(\eta_t)$  in Theorem \ref{SUpqmarkov} is given by
$$L=\frac{1}{2} \Psi^{-1} \circ H_{T} \circ \Psi.$$
Oshima and Shimeno \cite{Oshima_Shimeno} goes into that direction but is not precise enough to obtain this conclusion. We will see that this hold true for $SU(p,q)$, but that the result is quite subtle. The reason is that on the one hand, there are many ground states for $H_{T}$ and on the other hand that $L$ is the generator of a process with values in a proper subcone of $\R^p$ when $p >1$. We don't know if this is true for $S0(p,q)$ or $Sp(p,q)$.
\subsection{Asymptotics for $SU(p,q)$} We compute the generator of the process $\eta_t$ for $SU(p,q)$, using that in this case the spherical functions are explicitly known for all $(p,q)$. In $SO(p,q)$ and $Sp(p,q)$, such an expression is not known up to now.

For $\lambda\in \R^p$, the Harish Chandra spherical function $\varphi_\lambda^{(p,q)}$ of $SU(p,q)$, is defined by
(see \cite[Theorem IV.4.3]{Helgason}), for $g \in SU(p,q)$,
$$\label{eq_phi0}\varphi_\lambda^{(p,q)}(g)=\int_{ K}e^{(\lambda+\varrho)(H(kg))}\,dk.$$
Let $$A(p,q)=(-1)^{\frac{1}{2}(p(p-1)}2^{2p(p-1)}\prod_{j=1}^{p-1}\{(q-p+j)^{p-j}j!\},$$
$$c(p,q) =A(p,q)(-1)^{\frac{1}{2}p(p-1)}{(2!4!\cdots (2(p-1))!)^{-1}}.$$
By Hoogenboom \cite[Theorem 3]{Hoogenboom}, if for all $i$,   $\lambda_i\not \in \Z$, then when $r\in \R^p$, \begin{equation}\label{hoog}\varphi_\lambda^{(p,q)}(D_p(r))=\frac{A(p,q)\det(\varphi_{\lambda_i}^{(1,q-p+1)}(D_1(r_j))}{\prod_{1\leqÊi<j \leq p}(\cosh 2r_i- \cosh 2r_j)(\lambda_i^2-\lambda_j^2)}.\end{equation}
\begin{lemma} For $r\in \R^p$,
 $$\varphi_0^{(p,q)}(D_p(r))=\frac{c(p,q)}{\prod_{1\leq i<j \leq p}(\cosh 2r_i- \cosh 2r_j)}\det(M_{(p,q)}(r)),$$
where $M_{(p,q)}(r)$ is the $p\times p$ matrix with $(i,j)$ coefficient given by 
$$M_{(p,q)}(r)^{i,j}=\frac{d^{2(j-1)}}{d\lambda^{2(j-1)}}\varphi_{\lambda}^{(1,q-p+1)}(D_1(r_i))_{\{\lambda=0\}}.$$
\end{lemma}
\begin{proof} Since; for $\lambda\in \R$, $\varphi_{\lambda}^{(1,q-p+1)}=\varphi_{-\lambda}^{(1,q-p+1)}$, 
the lemma follows from Hua's lemma \cite[Lemma 4.1]{Hoogenboom}.
\end{proof}
For $SU(p,q)$, $m_\alpha=2(q-p)$ when $\alpha=e_k^*$, $m_\alpha=1$ when $\alpha=2e_k^*$ and $m_\alpha=2$ when $\alpha=e_i^*+e_j^*$ or $\alpha= e_i^*-e_j^*.$ 
Thus, by (\ref{rho}),
$$\delta(D_p(r))=\prod_k (2\sinh r_k)^{2(q-p)}(2 \sinh 2r_k) 
\prod_{i < j}4 (\cosh 2r_i-\cosh 2r_j)^2.$$
On the other hand,  for $r \in \R$, let $$\delta_{q-p+1}(r)=2^{2(q-p)+1}\sinh^{2(q-p)} r\sinh{2r}. $$
(this is (\ref{delta_q}) adapted to the complex case). Then there is  $C(p,q)>0$ such that, for $r\in \R^p$,
\begin{align*}
 (\delta^{1/2}\varphi_0^{(p,q)})(D_p(r))=C(p,q)\det M_{(p,q)}(r) \prod_{k=1}^p\delta_{q-p+1}(r_k) = C(p,q) \det N_{(p,q)}(r),
\end{align*}
where
$$(N_{(p,q)}(r))^{i,j}= \frac{d^{2(j-1)}}{d\lambda^{2(j-1)}}(\delta_{q-p+1}^{1/2}\tilde\varphi_{\lambda}^{q-p+1})(r_i)_{\{\lambda=0\}},$$
and where $\tilde\varphi_{\lambda}^{q-p+1}(s)=\varphi_{\lambda}^{(1,q-p+1)}(D_1(s))$ when $s\in \R$.
Let (see (\ref{alpha_q}),(\ref{g_q})) $$a(q-p+1)=\frac{\Gamma(q-p)}{\Gamma(2(q-p))2^{5/2}},$$ and
    $$g_{q-p+1}(\lambda,r)=a(q-p+1)(\delta_{q-p+1}^{1/2}\tilde\varphi_{\lambda}^{q-p+1})(r+\log 2(q-p))-K_{\lambda}(e^{-r}).$$
By Corollary \ref{cor_append}, $g_{q-p+1}(\lambda,r)$ and all its derivatives  at $\lambda=0$ converge to $0$  as $q\to +\infty$. 
The generator of $\mbox{Rad}(\xi_t^0)$ is
$$\frac{1}{2}\sum_{i=1}^p \frac{\partial^2}{\partial r_i^2} + \sum_{i=1}^p \frac{\partial}{\partial r_i} \log  (\delta^{1/2}\varphi_0^{(p,q)})(D_p(r))\frac{\partial}{\partial r_i}.$$
Recall that $$\eta_t=\mbox{SingVal}(l^{-1}_t\int_0^t l_sl^{*}_s  \, ds).$$
We deduce from Corollary \ref{cor_append} that:

\begin{theorem}\label{Theo_gen_lim}
The generator of the process $\log \eta_t, t \geq0,$ in Theorem \ref{SUpqmarkov} is
$$\frac{1}{2}\sum_{i=1}^p \frac{\partial^2}{\partial r_i^2} + \sum_{i=1}^p \frac{\partial}{\partial r_i} \log  \tilde K^{(p)}(e^{-r})\frac{\partial}{\partial r_i}, $$
where $\tilde K^{(p)}(e^{-r})$ is the determinant of the $p\times p$ matrix with $(i,j)$coefficient  $$\frac{d^{2(j-1)}}{d\lambda^{2(j-1)}}K_\lambda(e^{-r_i})_{\{\lambda=0\}}.$$
\end{theorem}
\begin{remark} The function $r\mapsto  \tilde K^{(p)}(e^{-r})$ is a ground state of the Toda Hamiltonian $$\sum_{k=1}^p \frac{\partial^2}{ \partial{r_k^2}}-e^{-2r_k}$$  equal to $0$ on the walls $\{r_i=r_j\}$. 
Notice that $\prod_{k=1}^p K_0(e^{-r_k})$ is another ground state.
\end{remark}

\begin{remark} Exactly as in Proposition \ref{condit}, but using (\ref{hoog}), one has that, for $r=\log \eta_t$,
           $$\E(e^{\sum_{k=1}^{p}\lambda_k l_t^{(k,k)}}| \sigma(\eta_s, 0\leqÊ  s \leq t))=\frac{c_p \tilde K_\lambda(e^{-r})
       }{\tilde K^{(p)}(e^{-r})\prod_{i<j}(\lambda_i^2-\lambda_j^2)},$$   
   where    $ c_p=(-1)^{\frac{1}{2}p(p-1)}(2!4!\cdots (2(p-1))!$ and $\tilde K_\lambda(e^{-r})$ is the determinant of the matrix $(K_{\lambda_i}(e^{-r_j}))$. As in Remark \ref{cameron}, this implies that when the diagonal part of $(l_t)$ has a drift $\lambda$, then $\log \eta_t$ is a Markov process with generator given by
$$\frac{1}{2}\sum_{i=1}^p \frac{\partial^2}{\partial r_i^2} + \sum_{i=1}^p \frac{\partial}{\partial r_i} \log \tilde K_\lambda(e^{-r})\frac{\partial}{\partial r_i}. $$
 \end{remark}

\begin{remark} After the submission of this paper, Rider and Valko \cite{Rider} have  posted a paper where they consider a Brownian motion $M_t$ on $Gl(p,\R)$, solution of the stochastic differential equation $dM_t=M_tdB_t+(\frac{1}{2}+\mu)M_t\, dt$ where $B_t$ is the $p\times p$ matrix made of $p^2$ independent standard real Brownian motions.
They show in particular, using a similar approach as Matsumoto and Yor, that $Z_t=M^{-1}_t\int_0^t M_s M_s^*\, ds, t \geq 0,$ is a Markov process if $|\mu| > (p-1)/2$ and describe its generator. For $Gl(p,\C)$, when $\mu=0$, $\mbox{SingVal}(Z_t)=\mbox{SingVal}(\eta_t)$, where  $\eta_t$ is given in the preceding remark with $\lambda=\varrho$.\end{remark}
 
\section{Series of homogeneous trees}

We now consider the same question for $q$-adic symmetric spaces of rank one.
 The most important series is given by the symmetric spaces associated with $Gl(2,\Q_q)$ where $q$ is the sequence of prime numbers which are  the homogeneous trees $\T_q$. The analogue of the Brownian motion is the simple random walk.  
 Let us consider more generally the tree $\T_q$, where $q$ is an arbitrary integer. We will deal to this case by an elementary treatment. By definition, $\T_q$ is the connected graph without cycle whose vertices have exactly $q + 1$ edges. We choose an origin $o$ in this tree. The simple random walk $W_n, n \geq 0$, starts at $o$ and at each step goes to one of its $q+1$ neighbours with uniform probability.

 Let us recall some  (well known) elementary facts about $(W_n)$ (\cite{Figa-Talamanca}, \cite{Woess}). For the convenience of the reader we give the simple proofs.
Let us write  $x\sim y$ when $x$ and $y$ are neighbours. The probability transition of $W_n$ is 
$$P(x,y)=\frac{1}{q+1}, \mbox{iff } x\sim y.$$
Let $d$ be the standard distance on the tree. The radial part of $W_n$ is the process $X_n=d(o,W_n)$, it is a Markov chain on $\N$ with transition probability $R$ given by $R(0,1)=1$, and if $n >0$,
 $$R(n,n-1)=\frac{1}{q+1}, R(n,n+1)=\frac{q}{q+1}.$$ A function $f:\T_q\to \R$ is called radial when 
$f(x)$ depends only on $d(o,x)$. In this case one defines
$\tilde f: \N \to \R$ by
$$\tilde f(n)=f(x), \mbox{ when } d(o,x)=n.$$

  Let us introduce the average operator $A$ which associates to a function $f:\T_q\to \R$ 
 the radial function $Af:\T_q\to \R$ defined by
 $$Af(x)=\frac{1}{S(o,x)}\sum_{y\in S(o,x)}f(y),$$
 where $S(o,x)$ is the sphere $S(o,x)=\{y\in \T_q; d(o,y)=d(o,x)\}$.
 It is easy to see that $PA=AP$, which implies that,
  \begin{lemma} \label{lem_aver}When a function $f:\T_q\to \R$ is a $\lambda$-eigenfunction of $P$ (i.e. $Pf=\lambda f$), then $Af$ is a radial  $\lambda$-eigenfunction of $P$  and $\widetilde {Af}$ is a $\lambda$-eigenfunction of $R$. Conversely, if 
 $\tilde g:\N \to \R$ is a $\lambda$-eigenfunction of $R$, then $ g(x)=\tilde g(d(o,x))$ is a radial $\lambda$-eigenfunction of $P$. \end{lemma}
 Let $\lambdaÊ>0$, $\tilde g:\N\to \R^+$ is a $\lambda$-eigenfunction of $R$ when $\tilde g(1)=\lambda \tilde g(0)$ and
 $$ \tilde g(n-1)+q \tilde g(n+1)=\lambda(q+1) \tilde g(n), n \geq 1.$$
 One sees easily that there exists a positive solution of this equation if and only if $\lambda \geqÊ\frac{2\sqrt{q}}{q+1}$. Hence, the  principal generalized eigenvalue $ \varrho$ of $R$ is
 $$\varrho = \frac{2\sqrt{q}}{q+1},$$
 the associated eigenfunction $\tilde \varphi_0$  is 
 $$\tilde \varphi_0(n)=(1+n\frac{q-1}{q+1})\frac{1}{q^{n/2}}.$$
 It is the only one if we suppose that $\tilde \varphi_0(0)=1$. It follows from Lemma \ref{lem_aver} that $ \varrho$ is also the principal generalized eigenvalue of $P$. 
The function $\varphi_0:\T_q\to \R$ given by  $\varphi_0(x)= \tilde \varphi_0(d(o,x))$ for $x\in \T_q$, is a radial $\varrho$-eigenfunction of $P$. By uniqueness of $\tilde \varphi_0$, $\varphi_0$ is the unique radial ground state of the random walk $W_n,n \geq 0$, equal to 1 at $o$.
  We consider the $\varphi_0$-ground state process $W^{(0)}_n, n \geq 0,$ on $\T_q$, defined as the Markov chain with transition probability $P^{(0)}$ given by, for $x,y \in \T_q$,
  $$P^{(0)}(x,y)=\frac{1}{\varrho \phi_0(x)}P(x,y)\phi_0(y).$$ Its radial part is a Markov chain on $\N$ with transition probability, for $m,n \in \N$,   $$R^{(0)}(n,m)=\frac{1}{\varrho\tilde \varphi_0(n)}R(n,m)\tilde \varphi_0(m).$$
  It is clear that
  $$\lim_{q \to +\infty}R^{(0)}(m,n)=B(m,n),$$
  where $B$ is transition probability of the so-called discrete Bessel(3) Markov chain on $\N$, namely,
  $$B(0,1)=1, B(n,n+1)=\frac{n+2}{2(n+1)}, B(n,n-1)=\frac{n}{2(n+1)}.$$
Therefore,
 \begin{proposition}\label{prop_conv_bes}
 When $q\to +\infty$, the radial part of the $\varphi_0$-ground state process $W^{(0)}_n$ of the simple random walk on the tree $\T_q$ converges in distribution to the discrete Bessel(3) chain.
 \end{proposition}
 
  Our aim is to obtain a path description of this limit process in terms of the simple random walk on $\Z$. We use the analogue of the distinguished Brownian motion. 
  A geodesic ray $\omega$  is an infinite sequence $x_n, n \geq 0,$ of distinct vertices in $\T_q$ such that $d(x_n,x_{n+1})=1$ for any $n \geqÊ0$. 
 We fix such a geodesic ray $\omega=\{x_n , n \geq 0\}$ starting from $x_0=o$. For any $x$ in $\T_q$ let $\pi_\omega(x)$ be the projection of $x$ on $\omega$, defined by $d(x,\pi_\omega(x))=\min\{d(x,y), y \in \omega\}$ and $\pi_\omega(x)\in \omega$.  The height $h(x)$ of $x\in \T_q$ with  respect to $\omega$ is
$$h(x)=d(x,\pi_\omega(x))-d(o,\pi_\omega(x)),$$
($h$ is a Busemann function). For $n \in \Z$ the horocycle $H_n$ is 
$$H_n=\{x\in \T_q, h(x)=n\}.$$

\begin{figure}
\setlength{\unitlength}{0.7mm}

\centering
\framebox{\begin{picture}(100,100)(-4,-5)

\multiput(9,20)(2,0){43}{\line(1,0){1}}
\multiput(19,40)(2,0){38}{\line(1,0){1}}
\multiput(29,60)(2,0){33}{\line(1,0){1}}
\multiput(39,80)(2,0){28}{\line(1,0){1}}

\put(16,42){$o$}
\put(-1,20){$H_{1}$}
\put(8,40){$H_{0}$}
\put(18,60){$H_{-1}$}
\put(28,80){$H_{-2}$}

\put(0,0){\vector(1,2){45}}
\put(47,90){$\omega$}
\put(2,-2){$\omega'$}

\put(40,80){\circle*{2}}
\put(40,80){\line(2,-1){40}}
\put(40,80){\line(1,-1){20}}

\put(70,80){\circle*{2}}
\put(70,80){\line(2,-1.3){22}}
\put(70,80){\line(2,-2){20}}
\put(70,80){\line(2,-1.6){22}}
\put(70,80){\line(-2,2){10}}

\put(90,80){\circle*{2}}
\put(90,80){\line(2,-1.3){5}}
\put(90,80){\line(2,-2){5}}
\put(90,80){\line(2,-1.6){5}}
\put(90,80){\line(-2,1.3){14}}

\put(30,60){\circle*{2}}
\put(60,60){\circle*{2}}
\put(80,60){\circle*{2}}
\put(90,60){\circle*{2}}
\put(60,60){\line(2,-1.3){30}}
\put(60,60){\line(2,-2){20}}
\put(60,60){\line(2,-1.6){25}}
\put(80,60){\line(1,-0.5){14}}
\put(80,60){\line(1,-1){14}}
\put(80,60){\line(1,-0.7){14}}
\put(30,60){\line(2,-1){40}}
\put(30,60){\line(1,-1){20}}

\put(20,40){\circle*{2}}
\put(20,40){\line(2,-1){40}}
\put(20,40){\line(1,-1){20}}
\put(50,40){\circle*{2}}
\put(50,40){\line(2,-2){20}}
\put(50,40){\line(2,-1.3){30}}
\put(50,40){\line(2,-1.6){25}}
\put(70,40){\circle*{2}}
\put(70,40){\line(2,-1.5){24}}
\put(70,40){\line(2,-2.2){18}}
\put(70,40){\line(2,-1.8){22}}
\put(80,40){\circle*{2}}
\put(80,40){\line(2,-2){14}}
\put(80,40){\line(2,-1.4){14}}
\put(80,40){\line(2,-1.7){14}}
\put(85,40){\circle*{2}}
\put(85,40){\line(2,-1.2){10}}
\put(85,40){\line(2,-1.4){10}}
\put(85,40){\line(2,-1.6){10}}

\put(90,40){\circle*{2}}
\put(90,40){\line(2,-1){5}}
\put(90,40){\line(2,-2){5}}
\put(90,40){\line(2,-1.5){5}}

\put(50,40){\circle*{2}}
\put(50,40){\line(2,-2){20}}
\put(50,40){\line(2,-1.3){30}}
\put(70,40){\circle*{2}}
\put(70,40){\line(2,-1.5){24}}
\put(70,40){\line(2,-2.2){18}}

\put(40,20){\circle*{2}}
\put(40,20){\line(1,-1.1){18}}
\put(40,20){\line(1,-0.8){25}}
\put(40,20){\line(1,-0.9){22}}
\put(60,20){\circle*{2}}
\put(60,20){\line(1,-1.2){16}}
\put(60,20){\line(1,-1){20}}
\put(60,20){\line(1,-1.1){18}}
\put(70,20){\circle*{2}}
\put(70,20){\line(2,-2){20}}
\put(70,20){\line(2,-2.2){18}}
\put(70,20){\line(2,-2.5){16}}
\put(75,20){\circle*{2}}
\put(75,20){\line(2,-2){20}}
\put(75,20){\line(2,-2.2){18}}
\put(75,20){\line(2,-1.8){20}}
\put(80,20){\circle*{2}}
\put(80,20){\line(2,-1.5){15}}
\put(80,20){\line(2,-2){15}}
\put(80,20){\line(2,-1.7){15}}
\put(88,20){\circle*{2}}
\put(88,20){\line(2,-1.5){7}}
\put(88,20){\line(2,-1.7){7}}
\put(88,20){\line(2,-2){7}}

\put(92,20){\circle*{2}}
\put(92,20){\line(2,-1.4){3}}
\put(92,20){\line(2,-1.7){3}}
\put(92,20){\line(2,-2.1){3}}

\put(10,20){\circle*{2}}
\put(10,20){\line(2,-1){40}}
\put(10,20){\line(1,-1){20}}
\end{picture}}
\caption{Example $\T_3$}\label{ExampleT3}
\end{figure}

As described by Cartier \cite{Cartier} (see also \cite{Cartwright}), $\T_q$ can be viewed as an infinite genealogical tree with $\omega$ as the unique "mythical ancestor". Then $H_n$ is the $n$-th generation. Any vertex (individual) $x$ in $\T_q$ has $q$ neighbours of height $h(x)+1$ (his children) and one neighbour of height $h(x)-1$ (his parent). This implies that 
if $W_n$ is the simple random walk on $\T_q$, then  $h(W_n)$ is the random walk on $\Z$ with transition probability $H$ given by, for any $n \in \Z$,  
 $$H(n,n-1)=\frac{1}{q+1},  H(n,n+1)=\frac{q}{q+1}. $$
A positive eigenfunction $f$ of this kernel is  a solution of 
$$ f(n-1)+q f(n+1)=\lambda(q+1) f(n), n \in \Z, f(0)=1,$$
and one sees easily that it exists if and only if  $$\lambda \geq \frac{2\sqrt{q}}{q+1}=\varrho.$$ As a result $\varrho$ is also the principal generalized eigenvalue (or spectral radius) of $H$. Moreover the associated eigenfunction is
$f(n)={q^{-n/2}}, n\in \N.$
This implies that  the function 
$$\varphi_\omega(x)=q^{-h(x)/2}, \, x \in \T_q,$$
 is, like $\varphi_0$, a ground state of the Markov chain $W_n$. By Lemma \ref{lem_aver},
 $A\varphi_\omega$ is a radial ground state of $P$. By uniqueness, this implies that 
\begin{equation}\label{eq_phi_omega}\varphi_0=A\varphi_\omega.\end{equation}
\begin{definition} Let $S_n,n \geqÊ0,$ be the $\varphi_\omega$-ground state process associated to $(W_n)$ on $\T_q$ starting from $o$. We denote by $Q_\omega$ its probability transition.
\end{definition}
We have,
\begin{align*}
Q_\omega(x,y)&=P(x,y)\frac{\varphi_\omega(y)}{\varrho \varphi_\omega(x)}
=P(x,y)\frac{q^{-h(y)/2}}{\varrho q^{-h(x)/2}}
=P(x,y)\frac{(q+1)q^{h(x)-h(y)/2}}{2\sqrt{q} }.\end{align*}
Therefore, when $x\sim y$, since $P(x,y)=1/(q+1)$,
\begin{align*}
  Q_\omega(x,y)=
  \begin{cases}
    1/2q, \mbox{ when } h(y)=h(x)+1,\\
    1/2, \mbox{ when } h(y)=h(x)-1.
  \end{cases}
\end{align*}
As in Proposition \ref{same_law},
\begin{lemma} \label{samelaw} The process $d(o,S_n), n \geqÊ0,$ has the same law as $d(o,W^{(0)}_n)$, $nÊ\geq 0$.
\end{lemma}

\begin{proof} Let $N\in \N$ and  $F:\Z^{N+1}\to \R^+$. For any isometry $k$ of the tree $\T_q$ which fixes $o$, $W_n, n\geq0$, has the same law as $k.W_n, nÊ\geq0$. Hence, using also (\ref{eq_phi_omega}),
\begin{align*}
\E(F(d(o,S_n), n \leqÊN))&=\varrho^{-N}\E(F(d(o,W_n), n \leqÊN)\varphi_\omega(W_N))\\
&=\varrho^{-N}\E(F(d(o,k.W_n), n \leqÊN)\varphi_\omega(k.W_N))\\
&=\varrho^{-N}\E(F(d(o,W_n), n \leqÊN)A\varphi_\omega(W_N))\\
&=\varrho^{-N}\E(F(d(o,W_n), n \leqÊN)\varphi_0(W_N))\\
&=\E(F(d(o,W^{(0)}_n), n \leqÊN)). 
\end{align*}\end{proof}

\begin{figure}
\centering
\framebox{\begin{picture}(80,50)(4,4)

\put(55,52){$\omega$}

\multiput(5,30)(2,0){40}{\line(1,0){1}}
\multiput(30,4)(0,2){26}{\line(0,1){1}}
\put(25,30){$\scriptstyle O$}
\put(31,52){$\scriptstyle  Oy$}
\put(78,31){$\scriptstyle Ox$}


\put(30,54.5){\vector(0,1){1}}

\put(84,30){\vector(1,0){1}}

\put(6,4){$\omega'$}

\put(10,10){\circle*{2}}
\put(20,20){\circle*{2}}
\put(30,30){\circle*{2}}
\put(40,40){\circle*{2}}
\put(50,50){\circle*{2}}

\thicklines

\put(30,10){\circle{2}}
\put(50,10){\circle{2}}
\put(70,10){\circle{2}}

\put(40,20){\circle{2}}
\put(60,20){\circle{2}}
\put(80,20){\circle{2}}

\put(30,30){\circle{2}}
\put(50,30){\circle{2}}
\put(70,30){\circle{2}}

\put(60,40){\circle{2}}
\put(80,40){\circle{2}}

\put(70,50){\circle{2}}

\put(10,10){\line(1,-1){6}}
\put(20,20){\line(1,-1){16}}
\put(30,30){\line(1,-1){26}}
\put(40,40){\line(1,-1){36}}
\put(50,50){\line(1,-1){34}}
\put(66,54){\line(1,-1){18}}

\put(4,4){\vector(1,1){50}}
\end{picture}}

 \centering \caption{The graph $G$}
 \label{GraphG} 
\end{figure}
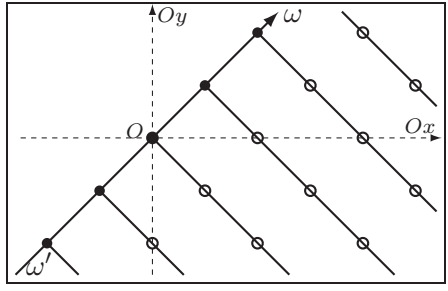

We choose in the tree $\T_q$ another geodesic ray $\omega'=\{x_{-n}, n \geq 0\}$ such that $\omega \cap \omega'=\{o\}$. Then $\omegaÊ\omega'=\{x_n, n\in \Z\}$  is a two-sided geodesic. We consider the following (unoriented) graph $G$ embedded in $\Z^2$ (see Figure \ref{GraphG}): 
the edges of $G$ are the points with coordinates
$$(k +2n, k), k \in \Z, n \in \N,$$
and the vertices are the first diagonal segments joining $(k,k)$ and $(k+1,k+1)$ and the segments joining  $(k+2n, k)$ and $(k+2n+1, k-1)$ for $n \geq 0$. It is obtained from the tree by gathering siblings (with one exception on each vertex of $\omegaÊ\omega'$).   Let $\Psi:\T_q \to G$ defined as follows: for $x \in \T_q$, 
$$\Psi(x)= (-h(x)+2d(x, \omega' \omega),-h(x)).$$
Then $\Psi(S_n), n \geqÊ0, $ is the nearest neighbour Markov chain  on the graph $G$
with probability transition $P_G$ given by
\begin{align*}&P_G((k,k),(k-1,k-1))=1/2q,\\
&P_G((k,k),(k+1,k+1))=1/2,\\
&P_G((k,k), (k+1,k-1)=(q-1)/2q,
\end{align*}
and, for $n >0$ and $\varepsilon=1$ or $-1$,
$$P_G((k+2n,k), (k+2n-\varepsilon,k+\varepsilon))=1/2.$$
When $q\to +\infty$, the Markov chain $\Psi(S_n)$ converges to the Markov chain $\tau_n, n \in \N,$ on the graph $G$ with transition probability $Q$ given by 
\begin{align*}&Q((k,k),(k-1,k-1))=0,\\ &Q((k,k),(k+1,k+1))=1/2,\\
&Q((k,k), (k+1,k-1)=1/2,\end{align*}
and for $n >0$ and $\varepsilon=1$ or $-1$,
$$Q((k+2n,k), (k+2n-\varepsilon,k+\varepsilon))=1/2.$$
When the chain $\tau_n$  starts from $(0,0)$ it cannot reach $(-1,-1)$ therefore it lives on the subgraph $\tilde G$ of $G$ described in  Figure \ref{Pitmanswalk}. It is the one which appears in the discrete time Pitman theorem (see Figure 8 in Biane \cite{Biane}). 
It  has the following simple description in terms of the simple symmetric random walk $\Sigma_n,n \geqÊ0,$ on $\Z$. Let $\Sigma_n=\varepsilon_1+\cdots+\varepsilon_n$ where the $(\varepsilon_n)$ are i.i.d. random variables such that $\P(\varepsilon_n=1)=\P(\varepsilon_n=-1)=1/2$ and $$M_n=\max\{\Sigma_k, k \leq n\}.$$ The Markov chain $\tau_n$ on $\tilde G$ starting from $(0,0)$ is, in $xOy$ coordinates, $$\tau_n=(2M_n-\Sigma_n, \Sigma_n), $$ (see Figure \ref{Pitmanswalk}). The radial component of $S_n$ is $d(o,S_n)= \delta((0,0), \Psi(S_n))$ for the graph distance $\delta$ on $ G$. Since $\delta((0,0), (a,b))=a$, we obtain using Lemma \ref{samelaw}   that,
\begin{theorem} The limit as $q\to +\infty$ of the processes $d(o,S_n), n \geq 0,$ and $d(o,W_n^{(0)}), n\geq0,$ have the same law as $$2\max_{0Ê\leqÊk \leqÊn}\Sigma_k-\Sigma_n, n \geqÊ0,$$ where $\Sigma_n$ is the simple symmetric random walk on $\Z$.
\end{theorem}
By Proposiiton \ref{prop_conv_bes}, we recover the following theorem of Pitman.
\begin{corollary}[\cite{Pitman}]
The process $2\max_{0Ê\leqÊk \leqÊn}\Sigma_k-\Sigma_n, n \geqÊ0,$ is the discrete Bessel(3) Markov chain.
\end{corollary}
\setlength{\unitlength}{1mm}

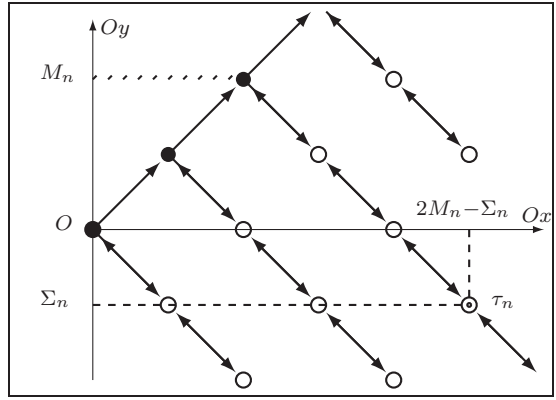
\begin{figure}
\centering
\framebox{\begin{picture}(70,50)(20,9)

\put(30,30){\vector(1,0){60}}
\put(25,30){$\scriptstyle O$}
\put(31,56){$\scriptstyle  Oy$}
\put(87,31){$\scriptstyle Ox$}
\put(30,10){\vector(0,1){48}}
\put(30,30){\circle*{2}}
\put(40,40){\circle*{2}}
\put(50,50){\circle*{2}}
\thicklines
\put(50,10){\circle{2}}
\put(70,10){\circle{2}}
\multiput(30,50)(2,0){10}{\line(1,1){0.5}}
\multiput(30,20)(2,0){25}{\line(1,0){1}}
\multiput(80,21)(0,2){5}{\line(0,1){1}}
\put(83,20){$\scriptstyle \tau_n$}
\put(40,20){\circle{2}}
\put(60,20){\circle{2}}
\put(23,50){$\scriptstyle M_n$}
\put(23,20){$\scriptstyle \Sigma_n$}
\put(73,32){$\scriptstyle 2M_n-\Sigma_n$}
\put(80,20){\circle{2}}
\put(80,20){\circle{.5}}
\put(30,30){\circle{2}}
\put(50,30){\circle{2}}
\put(70,30){\circle{2}}
\put(60,40){\circle{2}}
\put(80,40){\circle{2}}
\put(70,50){\circle{2}}
\put(30,30){\vector(1,1){9}}
\put(40,40){\vector(1,1){9}}
\put(50,50){\vector(1,1){9}}
\put(30,30){\vector(1,-1){9}}
\put(40,40){\vector(1,-1){9}}
\put(50,50){\vector(1,-1){9}}
\put(61,59){\vector(1,-1){8}}
\put(42,18){\vector(1,-1){7}}
\put(52,28){\vector(1,-1){7}}
\put(62,38){\vector(1,-1){7}}
\put(72,48){\vector(1,-1){7}}
\put(38,22){\vector(-1,1){7}}
\put(48,32){\vector(-1,1){7}}
\put(58,42){\vector(-1,1){7}}
\put(68,52){\vector(-1,1){7}}
\put(62,18){\vector(1,-1){7}}
\put(72,28){\vector(1,-1){7}}
\put(48,12){\vector(-1,1){7}}
\put(68,12){\vector(-1,1){7}}
\put(58,22){\vector(-1,1){7}}
\put(78,22){\vector(-1,1){7}}
\put(68,32){\vector(-1,1){7}}
\put(78,42){\vector(-1,1){7}}
\put(87,13){\vector(-1,1){6}}
\put(82,18){\vector(1,-1){7}}
\end{picture}}
  \caption{Pitman's walk on the graph $\tilde G$}
  \label{Pitmanswalk}
\end{figure}

\section{Appendix : Asymptotics of spherical functions in rank one} 
   We describe the needed asymptotic behaviour of spherical functions on $SO(1,q)$, $SU(1,q)$ and $Sp(1,q)$ when $q \to +\infty$. We adapt Shimeno \cite{Shimeno} to this setting (which only considers the real  split case) and Oshima and Shimeno \cite{Oshima_Shimeno}. The advantage of this approach is that it is adapted to higher rank cases. The details of the computations are quite long but straightforward. Therefore we only indicate the main points of the proof where they differ from \cite{Shimeno}. 
   
   We adapt the notions of Section \ref{Inozemtsev} to the rank one case. In this case there are at most two roots, $\alpha$ and $2\alpha$ and we may suppose that $\alpha(r)=r$ when $r\in \mathfrak{a}=\R$. Their multiplicity are $(m_\alpha,m_{2\alpha})=(q-1,0)$ for  $SO(1,q)$, $(2(q-1),1)$ for $SU(1,q)$ and $(4(q-1),3)$ for $Sp(1,q)$. 
   Let 
      $$\varrho_q=\frac{1}{2}(m_\alpha+2m_{2\alpha}),$$ 
      and \begin{equation}\label{delta_q}\delta_q(r)= (e^r-e^{-r})^{m_\alpha} (e^{2r}-e^{-2r})^{m_{2\alpha}}.\end{equation}
Then (see (\ref{eqn:sutherland}))
\begin{align*}
H_{CMS}&=\delta_q^{1/2}\circ\{\Delta_R  +\rho_q^2\}\circ \delta_q^{-1/2} \\ 
&=\frac{d^2}{ {dr^2}}-\frac{m_\alpha(m_\alpha+2m_{2\alpha}-2)}{ \sinh^2 r}-\frac{m_{2\alpha}( m_{2\alpha}-2)}{ \sinh^2 2r}.\end{align*}
For $\lambda\in \C$, the spherical function  $\varphi_\lambda$ satisfies
$$\Delta_R \tilde \varphi_\lambda=(\lambda^2-\varrho_q^2)\tilde \varphi_\lambda,$$
where $\tilde \varphi_\lambda(r)=\varphi_\lambda(D_1(r))$, $r\in \R$,
therefore,
$$H_{CMS} (\delta_q^{1/2}\tilde\varphi_\lambda)=\lambda^2\delta_q^{1/2}\tilde \varphi_\lambda.$$
There exists a unique function $\Psi_{CMS}(\lambda,q,r), r\in \R,$ of the form
\begin{equation}
\label{eqn:hhg}
\Psi_{CMS}(\lambda,q,r)=\sum_{n\in\N}b_n(\lambda,q)
e^{(\lambda-n)r}, \quad b_0(\lambda,q)=1,
\end{equation} such that 
$
H_{CMS}\Psi_{CMS}=\lambda^2\,\Psi_{CMS},
$
(see \cite[(17)]{Shimeno}). When $q\to +\infty$, 
  $$\lim H_{CMS}(r+\log m_\alpha)=H_T,$$
  where $H_T$ is the Toda type Hamiltonian
  $$H_T=\frac{d^2}{dr^2}-e^{-2r}.$$
There is also a unique function $\Psi_\text{T}(\lambda,r), r \in \R,$ of the form
\begin{equation}
\label{eqn:psitoda}
\Psi_T(\lambda,r)=\sum_{n\in \N}
b_n(\lambda)e^{(\lambda-n)r},\quad b_0(\lambda)=1,
\end{equation}
 such that 
   $H_T \Psi_T= \lambda^2 \Psi_T,$ (\cite{Shimeno}, notice that this is the function denoted $\Psi_T(-\lambda,-r)$ in \cite{Shimeno}).
   \begin{lemma} \label{prop1Shi}
If $\lambda\in\mathfrak{a}_\mathbb{C}^*$ and $2\lambda \not \in \Z^*$, then 
$$\lim_{q \to +\infty}Ê\frac{1}{m_\alpha^\lambda}\Psi_{CMS}(\lambda,q,r+\log m_\alpha)=\Psi_T(\lambda,r)$$   \end{lemma}
   \begin{proof} The proof is similar to the one of Proposition 1 in \cite{Shimeno}.
  \end{proof}
     Let \begin{equation}\label{alpha_q}a(q)=\frac{\Gamma({m_\alpha/2})^2}{\Gamma(m_\alpha)2^{1+3m_{2\alpha}/2}},\end{equation} and
    \begin{equation}\label{g_q}g_q(\lambda,r)=a(q)(\delta_q^{1/2}\tilde\varphi_\lambda)(\log m_\alpha +r)-K_{\lambda}(e^{-r}).\end{equation}
 \begin{proposition}\label{convsphe}
Uniformly on $\lambda$ in a small neighborhood of $0$ in $\C$, $\lambda g_q(\lambda,r)$ and its derivatives in $\lambda$ converge to $0$ as $q\to +\infty$.
\end{proposition}
       \begin{proof}
 Using the Harish Chandra spherical function expansion of $\varphi_\lambda$ (see \cite[Theorems IV.5.5 and IV.6.4]{Helgason} or \cite{Gangolli_Varadarajan}), when $\lambda$ is for instance in the ball $\{\lambda\in \C; |\lambda| \leqÊ1/4\}$ and $\lambda \not = 0$, one has   $$\delta_q^{1/2}(r)\tilde\varphi_\lambda(r)=c(\lambda)\Psi_{CMS}(\lambda,q,r)+c(-\lambda)\Psi_{CMS}(-\lambda,q,r),$$
    where     $$c(\lambda)=\frac{2^{\frac{1}{2}m_\alpha+m_{2\alpha}-\lambda}\Gamma(\frac{1}{2}(m_\alpha+m_{2\alpha}+1))}{
    \Gamma(\frac{1}{2}(\frac{1}{2}m_\alpha+1+\lambda)) \Gamma(\frac{1}{2}(\frac{1}{2}m_\alpha+m_{2\alpha}+\lambda))} .$$
    It follows from Lemma \ref{prop1Shi} that
    $$\lim_{q \to +\infty}Êa(q)c(\lambda)\frac{1}{m_\alpha^\lambda}\Psi_{CMS}(\lambda,q,r+\log m_\alpha)=\Gamma(\lambda)2^{\lambda-1}\Psi_T(\lambda,r).$$   
    Hence
  $$\lim_{q \to +\infty}a(q)(\delta_q^{1/2}\tilde\varphi_\lambda)(r+\log m_\alpha)=\Gamma(\lambda)2^{\lambda-1}   \Psi_T(\lambda,r)+\Gamma(-\lambda)2^{-\lambda-1}   \Psi_T(-\lambda,r).$$
The limit can be expressed in terms of the Whittaker function for $Sl(2,\R)$ as in \cite[Theorem 3]{Shimeno}. Due to the relation between Whittaker and Macdonald functions (see, e.g.,  Bump \cite{Bump}), one finds that this limit is $K_{\lambda}(e^{-r})$, thus $g_q(\lambda,r)$ tends to 0 when $q\to +\infty$.

Now we remark that, for $r$ fixed, the functions $ \lambda g_q(\lambda,r)$ are analytic functions in $\lambda$ in the ball $\{\lambda\in \C; |\lambda| \leqÊ1/4\}$ which are uniformly bounded in $q$ by the computations of  \cite[Proposition 1]{Shimeno}. Notice that we have to multiply by $\lambda$ to avoid the singularity at $0$ of the Harish Chandra $c$ function. 
   The uniform convergence of   $\lambda g_q(\lambda,r)$ and its derivatives in $\lambda$ thus follows from Montel's theorem.
      \end{proof}
      
 \begin{corollary} \label{cor_append} As $q\to +\infty$, $g_q(\lambda,r)$ and all its derivatives at $\lambda=0$ converge to $0$.
 \end{corollary}
 \proof This follows from the proposition and the fact that
 $$\frac{d^n}{d\lambda ^n}g_q(\lambda,r)_{\{\lambda=0\}}=\frac{1}{n+1}\frac{d^{n+1}}{d\lambda ^{n+1}}(\lambda g_q(\lambda,r))_{\{\lambda=0\}}.$$

\end{document}